\documentclass[a4paper]{amsart}

\usepackage{amscd,amsthm,amsfonts,amssymb,amsmath}
\usepackage{fullpage}
\usepackage[all]{xy}

    \newcommand{\BC}{{\mathbb {C}}} 
     \newcommand{\BF}{{\mathbb {F}}}

    \newcommand{\BQ}{{\mathbb {Q}}} \newcommand{\BR}{{\mathbb {R}}}
     \newcommand{\BT}{{\mathbb {T}}}

     \newcommand{\BZ}{{\mathbb {Z}}}

    \newcommand{\CE}{{\mathcal {E}}} 
     \newcommand{\CH}{{\mathcal {H}}}

    \newcommand{\CM}{{\mathcal {M}}} \newcommand{\CN}{{\mathcal {N}}}
     \newcommand{\CP}{{\mathcal {P}}}
     \newcommand{\CR}{{\mathcal {R}}}
     \newcommand{\CT}{{\mathcal {T}}}

    \newcommand{\fq}{{\mathfrak{q}}}

     \newcommand{\fH}{{\mathfrak{H}}}

     \newcommand{\fX}{{\mathfrak{X}}}

    \newcommand{\Div}{{\mathrm{Div}}}

    \newcommand{\Res}{{\mathrm{Res}}}

    \newcommand{\SL}{{\mathrm{SL}}}

    \newcommand{\sgn}{{\mathrm{sgn}}}

    \theoremstyle{plain}
    \newtheorem{thm}{Theorem}[section] \newtheorem{cor}[thm]{Corollary}
    \newtheorem{lem}[thm]{Lemma}  \newtheorem{prop}[thm]{Proposition}
    \newtheorem {conj}[thm]{Conjecture} \newtheorem{defn}[thm]{Definition}

\theoremstyle{remark} \newtheorem{remark}[thm]{Remark}
\theoremstyle{remark} 
\theoremstyle{remark} 

    \numberwithin{equation}{section}

\begin{document}

\title{On the torsion subgroups of the modular Jacobians}
\author{Yuan Ren}

\address{Academy of Mathematics and Systems Science, Chinese Academy of Sciences, Beijing, China}
\email{ry198628@163.com}

\begin{abstract}For any positive integer $N$, we prove that the rational torsion subgroup of $J_0(N)$ agrees with its rational cuspidal subgroups up to a factor of $6N\prod_{p\mid N}(p^2-1)$. Moreover, for modular Jacobians of the form $J_0(DC)$ with $D$ a positive square-free integer and $C$ any positive divisor of $D$, we prove that the $\psi$-part of the torsion subgroup of $J_0(DC)$ agrees with the $\psi$-part of its cuspidal subgroup up to a factor of $6D\prod_{p\mid D}(p^2-1)$, where $\psi$ is any quadratic character of conductor dividing $C$.
\end{abstract}

\maketitle

\tableofcontents
\section{Introduction}
For any positive integer $N$, let $X_0(N)$ be the canonical model over $\BQ$ of the modular curve of level $\Gamma_0(N)$ and let $J_0(N)$ be the Jacobian variety of $X_0(N)$ over $\BQ$. When $N=p$ is a prime, Ogg proved
\[C_0(p)\simeq\BZ/\frac{p-1}{(p-1,12)}\BZ,\]
where $C_0(p)$ is the cuspidal subgroup of $J_0(p)$ generated by the class of the divisor $[0]-[\infty]$ with $[0]$ and $[\infty]$ the two cusps of $X_0(p)$, and conjectured that
\[J_0(p)(\BQ)_{tor}=C_0(p)\]
(see \cite{Ogg1} and \cite{Ogg2}). Here $[0]-[\infty]$ is a $\BQ$-rational point of $J_0(p)$ since both $[0]$ and $[\infty]$ are $\BQ$-rational points in $X_0(p)$. In fact, for any positive integer $N$, the set of cusps of $X_0(N)$ is stable under the action of $G_\BQ$, and each positive divisor $d\mid N$ corresponds to a unique $G_\BQ$-orbit consisting of those cusps defined precisely over $\BQ(\mu_{(d,N/d)})$ (see \S1.3 of \cite{St}).

The above conjecture of Ogg has been proved by Mazur in his celebrated work \cite{M}, where the unique normalized weight-two Eisenstein series $E$ of level $\Gamma_0(p)$ plays a fundamental role. In fact, $C_0(p)$ is exactly the cuspidal subgroup associated to $E$ (see Definition~\ref{def}). Moreover, let $\BT_0(p)$ be the full Hecke algebra of level $\Gamma_0(p)$ generated over $\BZ$ by the Hecke operators $T_\ell$ for all the primes $\ell$. Then the action of $\BT_0(p)$ on $J_0(p)$ preserves $C_0(p)$ and induces an isomorphism
\[\BT_0(p)/I_{\Gamma_0(p)}(E)\simeq C_0(p),\]
where $I_{\Gamma_0(p)}(E)$ is the Eisenstein ideal of $E$ (see also Definition~\ref{def}). This isomorphism, which gives us the structure of $C_0(p)$ as a $\BT_0(p)$-module, is one of the key ingredient in the proof of Ogg's conjecture by Mazur. Here, we should remark that Mazur actually defined $\BT_0(p)$ to be the $\BZ$-algebra generated by all $T_\ell$'s with $\ell\neq p$ and the Atkin-Lehner operator $w_p$. But since $w_p=-T_p$ in this situation, these two definitions are in fact the same. After their pioneering work, one is naturally led to the following

\begin{conj}(Generalized Ogg's conjecture) For any positive integer $N$, we have that
\[J_0(N)(\BQ)_{tor}=C_0(N)(\BQ)\]
where $C_0(N)$ is the subgroup of $J_0(N)(\overline{\BQ})$ generated by degree zero divisor classes supported at the cusps of $X_0(N)$, and $C_0(N)(\BQ)=C_0(N)(\BQ)^{G_{\BQ}}$ is the $\BQ$-rational subgroup of $C_0(N)$.
\end{conj}

It is clear that the above conjecture is equivalent to $J_0(N)(\BQ)_{tor}\subseteq C_0(N)$ for any positive integer $N$. To this date, it has been proved that:

$\bullet$ If $p\geq5$ is a prime and $r\in\BZ_{\geq1}$, then $J_0(p^r)(\BQ)[q^{\infty}]\subseteq C_0(p^r)[q^{\infty}]$ for any prime $q\nmid6p$. See \cite{L}

$\bullet$ Let $N$ be a square-free positive integer, then we have $J_0(N)(\BQ)[q^{\infty}]=C_0(N)[q^{\infty}]$ for any prime $q\nmid6$ (See \cite{Oh}). Note that when $N$ is square free, all the cusps of $X_0(N)$ are in fact $\BQ$-rational and hence $C_0(N)=C_0(N)(\BQ)$.

The first main result of this article is the following

\begin{thm}\label{M1} For any positive integer $N$, we have that
\[J_0(N)(\BQ)[q^{\infty}]=C_0(N)(\BQ)[q^{\infty}]\]
for any prime $q\nmid6\cdot N\cdot\varpi(N)$, where $\varpi(N):=\prod_{p\mid N}(p^2-1)$.
\end{thm}

Our proof of this theorem is based on a careful study of modular Jacobian varieties of the form $J_0(DC)$, where $D$ is a positive square-free integer and $C$ is a positive divisor of $D$. In fact, in this situation, we can prove that the torsion points of $J_0(DC)$ over some quadratic fields also come from the cusps of $X_0(DC)$. Note that since the cusps of $X_0(DC)$ are all defined over $\BQ(\mu_C)$ as remarked before, so is the cuspidal subgroup $C_0(DC)$ of $J_0(DC)(\overline{\BQ})$. For any quadratic Dirichlet character $\psi$ of conductor $f_\psi\mid C$, we define
\[C_0(DC)(\psi):=\{P\in C_0(DC):\ \sigma(P)=\psi(\sigma)\cdot P \text{ for any }\sigma\in G_\BQ\},\]
and define similarly
\[J_0(DC)(\psi):=\{P\in J_0(DC)(\overline{\BQ}):\ \sigma(P)=\psi(\sigma)\cdot P \text{ for any }\sigma\in G_\BQ\}.\]
Then our second main result is the following

\begin{thm}\label{M2} Let $D$ be a positive square-free integer and $C$ a positive divisor of $D$. Then for any quadratic Dirichlet character $\psi$ of conductor $f_\psi\mid C$, we have that
\[J_0(DC)(\psi)[q^{\infty}]=C_0(DC)(\psi)[q^{\infty}]\]
for any prime $q\nmid6\cdot D\cdot\varpi(D)$.
\end{thm}

In our investigation, the relation between the weight two Eisenstein series and the cuspidal subgroup plays a very important role, so we will give a brief review of this relation in the second section. Then, in the third section, we construct a Hecke eigne-basis $\{E_{M,L,\psi}\}$ for the space $\CE_2(\Gamma_0(DC),\BC)$ of Eisenstein series of weight two and level $\Gamma_0(DC)$ (see Definition~\ref{key definition} and Proposition~\ref{eigen}). While all these Eisenstein series are interesting, we will in this article focus on the study of those $E_{M,L,\psi}$ with $\psi$ a \emph{quadratic} character. The associated group $C_{\Gamma_0(DC)}(E_{M,L,\psi})$ will be called as \emph{quadratic cuspidal subgroups} of $J_0(DC)$. The order and the Hecke module structure of these quadratic cuspidal subgroups are determined up to a factor of $6D$ (see Theorem~\ref{order} and Theorem~\ref{index}) in the fourth section. This will enable us to prove our main results in the final section.

Notations: For any positive integer $N=\prod_{p\mid N}p^{v_p(N)}$, we denote by $\varpi(N)=\prod_{p\mid N}(p^2-1)$, $\nu(N)=\sum_{p\mid N}v_p(N)$ and $\mu(N)=\prod_{p\mid N}(p+1)$. 

Let $\fq$ to be the function $z\mapsto e^{2\pi iz}$ on the upper half plane. For any function $g$ on the upper half plane and any $\gamma=\left(
                        \begin{array}{cc}
                          a & b \\
                          c & d \\
                        \end{array}
                      \right)
\in GL^+_2(\BR)$, we denote by $g|\gamma$ to be the function $z\mapsto det(\gamma)\cdot g(\gamma z)\cdot(cz+d)^{-2}$.

\section{Background materials}
In this section, we are going to recall the relation between weight two Eisenstein series and cuspidal subgroups. For more details and proof, the reader is referred to \cite{St} and \cite{St2}.

2.1. In the following, we fix a positive integer $N$ and denote by $\Gamma$ to be either $\Gamma_0(N)$ or $\Gamma_1(N)$. Let $\CM_2(\Gamma,\BC)$ be the space of weight two modular forms of level $\Gamma$, then
\[\CM_2(\Gamma,\BC)=S_2(\Gamma,\BC)\oplus\CE_2(\Gamma,\BC),\]
where $S_2(\Gamma,\BC)$ is the sub-space of cusp forms and $\CE_2(\Gamma,\BC)$ is the sub-space of Eisenstein series. For any positive integer $n$, there is a Hecke operator $\CT^{\Gamma}_n$ acting on $\CM_2(\Gamma,\BC)$ with respect to the above decomposition. We denote the restriction of $\CT^{\Gamma}_n$ to $S_2(\Gamma,\BC)$ by $T^{\Gamma}_n$. Let $\CT_\Gamma$ be the $\BZ$-algebra generated by $\{\CT^{\Gamma}_n\}_{n\geq1}$. Then the \emph{full} Hecke algebra $\BT_\Gamma$ of level $\Gamma$ is defined to be the restriction of $\CT_\Gamma$ to $S_2(\Gamma,\BC)$, which is the $\BZ$-algebra generated by all the $T_n$'s. When $\Gamma=\Gamma_0(N)$, we will also denote $\BT_{\Gamma_0(N)}$ as $\BT_0(N)$, which is in fact generated by the $T^{\Gamma_0(N)}_\ell$ for all the primes $\ell$.

2.2. Let $X_\Gamma$ be the modular curve over $\BQ$ of level $\Gamma$. We denote by $cusp(\Gamma)$ to be the set of cusps of $X_\Gamma$, and by $Y_\Gamma$ to be the complement of $cusp(\Gamma)$ in $X_\Gamma$. Let $J_\Gamma$ be the Jacobian variety of $X_\Gamma$ over $\BQ$. For any $g\in\CM_2(\Gamma,\BC)$, let $\omega_g$ be the meromorphic differential on $X_\Gamma(\BC)$ whose pullback to the Poincar$\acute{e}$ upper half-plane $\CH$ equals $g(z)dz$. The differential $\omega_g$ has all its poles supported at the cusps of $X_\Gamma$. Moreover, $g$ is a cusp form if and only if $\omega_g$ is holomorphic, or,  $\Res_x(\omega_g)=0$ for any $x\in cusp(\Gamma)$. Denote by $\Div^0(cusp(\Gamma);\BC)$ to be $\Div^0(cusp(\Gamma);\BZ)\otimes_\BZ\BC$, then we define the following homomorphism of $\BC$-vector spaces
\[\delta_\Gamma:\CE_2({\Gamma,\BC})\rightarrow \Div^0(cusp(\Gamma);\BC),\]
such that
\[E\mapsto 2\pi i\sum_{x\in{cusp(\Gamma)}}\Res_x(\omega_E)\cdot[x],\]
with $2\pi i\cdot\Res_x(\omega_E)=e_x\cdot a_0(E;[x])$, where $e_x$ is the ramification index of $X_\Gamma$ at $x$ and $a_0(g;[x])$ is the constant term of the Fourier expansion of $g$ at the cusp $x$. The homomorphism $\delta_\Gamma$ is actually an isomorphism by the theorem of Manin-Drinfeld. Because the restriction of $\omega_E$ to $Y_\Gamma$ is holomorphic, this differential induces the following periods integral homomorphism
\begin{align*}
\xi_E:H_1(Y_\Gamma(\BC),\BZ)\rightarrow\BC,\ [c]\mapsto\int_{c}\omega_E
\end{align*}
where $[c]$ is the homology class represented by a $1$-cycle $c$ on $Y_\Gamma(\BC)$. Note that, for any cusp $x$, we have
\begin{align*}
  \int_{c_x}\omega=2\pi i\cdot \Res_x(\omega_E),
\end{align*}
where $c_x$ is a small circle around $x$.

\begin{defn}\label{def}Let $E\in\CE_2(\Gamma,\BC)$ be a weight-two Eisenstein series of level $\Gamma$. We denote by $\CR_\Gamma(E)$ to be the sub-$\BZ$-module of $\BC$ generated by the coefficients of $\delta_\Gamma(E)$, and by $\CR(E)^{\vee}$ to be the dual $\BZ$-module of $\CR(E)$. Then :

(1) The cuspidal subgroup $C_\Gamma(E)$ associated with $E$ is defined to be the subgroup of $J_\Gamma(\overline{\BQ})$ which is generated by $\{w_\Gamma\left(\phi\circ\delta_\Gamma(E)\right)\}_{\phi\in\CR(E)^{\vee}}$, where $w_\Gamma$ is the Atkin-Lehner involution;
l
(2) The periods $\CP_\Gamma(E)$ of $E$ is defined to be the image of $\xi_E$. Since $\CP_\Gamma(E)$ contains $\CR_\Gamma(E)$ by the above remark, we can define $A_\Gamma(E)$ to be the quotient $\CP_\Gamma(E)/\CR_\Gamma(E)$;

(3) The Eisenstein ideal $I_\Gamma(E)$ of $E$ is defined to be the image of $Ann_{\CT_{\Gamma}}(E)$ in $\BT_{\Gamma}$.
\end{defn}

\begin{remark}
The above definition of $C_\Gamma(E)$ is slightly different from that given in \cite{St},as we have added an action of the Atkin-Lehner operator $w_\Gamma$. Since $w_\Gamma$ is an isomorphism, this modification does not change the order of the associated cuspidal subgroups. However, $C_\Gamma(E)$ is now annihilated by $I_\Gamma(E)$ under the usual action of the Hecke algebra, because $\CT^t_\ell\circ\delta_\Gamma=\delta_\Gamma\circ\CT_\ell$ and $\CT^t_\ell\circ w_\Gamma=w_\Gamma\circ\CT_\ell$ for any prime $\ell$.
\end{remark}

2.3. By Proposition 1.1 and Theorem 1.2 of \cite{St2}, $A_\Gamma(E)$ is finite and there is a perfect pairing $C_\Gamma(E)\times A_\Gamma(E)\rightarrow\BQ/\BZ$. Thus, the determination of the order of $C_\Gamma(E)$ is reduced to that of $\CP_\Gamma(E)$. In the following, we will recall a method due to Stevens for the computation of the periods. The reader is referred to \cite{St2} for details.

We first consider the case when $\Gamma=\Gamma_1(N)$. Denote by $S_N$ to be the set of all primes $p$ satisfying $p\equiv-1\pmod{4N}$. Let $\fX_N$ be the set of all non-quadratic Dirichlet character $\chi$ whose conductor is a prime in $S_N$, and $\fX^{\infty}_N$ be the set of all non-quadratic Dirichlet character $\chi$ whose conductor is of the form $p^M_\chi$ with $p_\chi\in S_N$ and $M$ some positive integer.

For any $E=\sum^{\infty}_{n=0}a_n(E;[\infty])\cdot \fq^n\in\CE_2(\Gamma_1(N),\BC)$ and any Dirichlet character $\chi$, the $L$-function associated to the pair $(E,\chi)$ is defined as
\begin{align*}
L(E,\chi,s):=\sum^{\infty}_{n=1}\frac{a_n(E;[\infty])\cdot\chi(n)}{n^s}.
\end{align*}
If $\chi\in\fX^{\infty}_N$ is of conductor $p^M_\chi$, then we define
\begin{align*}
  &\Lambda(E,\chi,1):=\frac{\tau(\overline{\chi})\cdot L(E,\chi,1)}{2\pi i},\\
  &\Lambda_{\pm}(E,\chi,1):=\frac{1}{2}(\Lambda(E,\chi,1)\pm\Lambda(E,\chi\cdot(\frac{}{p_\chi}),1)),
\end{align*}
where $(\frac{}{p_\chi})$ is the Legendre symbol associated to $p_\chi$. It is proved in Theorem 1.3 of \cite{St2} that, if $\CM$ is a finitely generated sub-$\BZ$-module of $\BC$, then the following are equivalent:

(St1) $\CP_{\Gamma_1(N)}(E)\subseteq\CM$;

(St2) $\CR_{\Gamma_1(N)}(E)\subseteq\CM$ and $\Lambda_{\pm}(E,\chi,1)\in\CM[\chi,\frac{1}{p_\chi}]$ for any $\chi\in\fX_{N}$;

(St3) $\CR_{\Gamma_1(N)}(E)\subseteq\CM$ and $\Lambda_{\pm}(E,\chi,1)\in\CM[\chi,\frac{1}{p_\chi}]$ for any $\chi\in\fX^{\infty}_N$.

Because $\Lambda_{\pm}(E,\chi,1)$ is essentially the Bernoulli numbers whose integrality and divisibility are well known (see Theorem 4.2 of \cite{St2}), we can then use the above result to determine the periods $\CP_{\Gamma_1(N)}(E)$ of $E$ and hence the order of $C_{\Gamma_1(N)}(E)$.

On the other hand, if $\Gamma=\Gamma_0(N)$, then Stevens' method can only determine $C_{\Gamma_0(N)}(E)$ up to its intersection with the Shimura subgroup. Recall that, if we denote by $\pi_N$ to be the natural projection of $X_1(N)$ to $X_0(N)$, then the Shimura subgroup of $J_0(N)$ is defined to be
\[\Sigma_N:=ker\left(\pi^*_N:\ J_0(N)\rightarrow J_1(N)\right),\]
which is a finite abelian group and is of multiplicative type as a $G_\BQ$-module. For any $E\in\CE_2(\Gamma_0(N),\BC)$, we define
\[A^{(s)}_{\Gamma_0(N)}(E):=\left(\CP_{\Gamma_1(N)}(E)+\CR_{\Gamma_0(N)}(E)\right)/\CR_{\Gamma_0(N)}(E),\]
then there is an exact sequence
\[
\xymatrix@C=0.5cm{
  0 \ar[r] & \Sigma_N\bigcap C_{\Gamma_0(N)}(E) \ar[r] & C_{\Gamma_0(N)}(E) \ar[r] & A^{(s)}_{\Gamma_0(N)}(E) \ar[r] & 0, }
  \]
which enables us to determine the order of $C_{\Gamma_0(N)}(E)/\Sigma_N\bigcap C_{\Gamma_0(N)}(E)$.

2.4. Finally, we recall some basic properties of the collection of functions $\{\phi_{\underline{x}}\}_{\underline{x}\in(\BQ/\BZ)^{\oplus2}}$ due to Hecke (see \cite{St}, Chapter 2, \S2.4) which we will need later. For any $\underline{x}=(x_1,x_2)\in(\BQ/\BZ)^{\oplus2}$, the Fourier expansion of $\phi_{\underline{x}}$ at infinity is
\begin{align}
  \phi_{\underline{x}}(z)+\delta(\underline{x})\cdot\frac{i}{2\pi(z-\overline{z})}=\frac{1}{2} B_2(x_1)-P_{\underline{x}}(z)-P_{-\underline{x}}(z)
\end{align}
for any $z\in\CH$, where $B_2(t)=\langle{t}\rangle^2-\langle{t}\rangle+\frac{1}{6}$ is the second Bernoulli polynomial and
\begin{align}
P_{\underline{x}}(z)=\sum_{k\in\BQ_{>0},k\equiv x_1(1)}k\sum^{\infty}_{m=1}e^{2\pi im(kz+x_2)}
\end{align}
and $\delta(\underline{x})$ is defined to be $1$ or $0$ according to $\underline{x}=0$ or not. If $\underline{x}\neq0$, then $\phi_{\underline{x}}$ is a (holomorphic) Eisenstein series. Moreover, for any $\underline{x}\in(\BQ/\BZ)^{\oplus2}$ and $\gamma\in SL_2(\BZ)$, we have
\begin{align}
  \phi_{\underline{x}}|\gamma=\phi_{\underline{x}\cdot\gamma}
\end{align}
where ${\underline{x}\cdot\gamma}$ is the natural right action of $\gamma$ on the row vector of length two. The whole collection of functions satisfy the following important \emph{distribution law}
\begin{align}
  \phi_{\underline{x}}=\sum_{\underline{y}:\ \underline{y}\cdot \alpha=\underline{x}}\phi_{\underline{y}}|\alpha
\end{align}
where $\alpha$ is any matrix in $M_2(\BZ)$ with positive determinant.

\section{An eigen-basis for $\CE_2(\Gamma_0(DC),\BC)$}
In this section, we will construct a basis for $\CE_2(\Gamma_0(DC),\BC)$ which plays a fundamental role in our later investigations. We will also show that the Eisenstein series in this basis are all eigenforms.

3.1. We will first introduce some operators on the $\BC$-vector space $\CM_2$ of weight-two holomorphic modular forms of all levels. For any prime $p$, we define an operator $\gamma_p$ on $\CM_2$ as following
\[\gamma_p:\CM_2\rightarrow\CM_2,\ g\mapsto g|\left(
                                                                                                 \begin{array}{cc}
                                                                                                   p & 0 \\
                                                                                                   0 & 1 \\
                                                                                                 \end{array}
                                                                                               \right).\]
If $\psi$ be a Dirichlet character of conductor $f_\psi$ and $p\nmid f_\psi$ is a prime, then we define the following two operators $[p]^{\pm}_\psi$ on $\CM_2$ as
\begin{align*}
[p]^+_\psi:&=1-\psi(p)\cdot\gamma_p\\
[p]^-_\psi:&=1-p^{-1}\cdot\psi^{-1}(p)\cdot\gamma_p
\end{align*}
More precisely, for any $g\in\CM_2 $ and any $z$ in the Poinc$\acute{a}$re upper half-plane $\CH$, we have that
\begin{align*}
[p]^+_\psi(g)(z)&=g(z)-p\cdot\psi(p)\cdot g(pz),\\
[p]^-_\psi(g)(z)&=g(z)-\psi^{-1}(p)\cdot g(pz).
\end{align*}
It is clear that if $p_1$ and $p_2$ are two primes not dividing $f_\psi$, then the four operators $[p_1]^+_\psi,[p_1]^-_\psi,[p_2]^+_\psi$ and $[p_2]^-_\psi$ are commutative with each other. Thus we can define, for any positive square-free integer $M$ prime to $f_\psi$, two operators $[M]^{\pm}_\psi$ on $\CM_2$ as
\[[M]^{\pm}_\psi:=[p_1]^{\pm}_\psi\circ[p_2]^{\pm}_\psi\circ...\circ[p_k]^{\pm}_\psi,\]
with $M=p_1\cdot p_2\cdot\cdot\cdot p_k$ in any order. When $\psi=1$ is the trivial Dirichlet character, we will write $[M]^{\pm}_\psi$ simply as $[M]^{\pm}$ for any positive square-free integer $M$.

\begin{remark}\label{zero}
It is easy to see that the above operators $[M]^{\pm}_\psi$ can also be applied to any function on $\CH$ in the same manner. In particular, we have that
\begin{align*}
[p]^+(\frac{1}{z-\overline{z}})=\frac{1}{z-\overline{z}}-\frac{p}{pz-p\overline{z}}=0,
\end{align*}
for any prime $p$. It follows that $[M]^+(\frac{1}{z-\overline{z}})=0$ for any square-free integer $M>1$.
\end{remark}

\begin{lem}\label{1.1}
Let $\psi$ be a Dirichlet character of conductor $f_\psi$, $p\nmid f_\psi$ be a prime and $N$ be a positive integer, then $[p]^{\pm}_\psi$ maps $M_2(\Gamma_0(N),\BC)$ to $M_2(\Gamma_0(Np),\BC)$ and satisfies the following properties

(1) For any prime $\ell\neq p$, we have that $\CT^{\Gamma_0(Np)}_\ell\circ[p]^{\pm}_\psi=[p]^{\pm}_\psi\circ\CT^{\Gamma_0(N)}_\ell$;

(2) If $p\nmid N$, then $\CT^{\Gamma_0(Np)}_p\circ[p]^{+}_\psi=\CT^{\Gamma_0(N)}_p-\gamma_p-p\cdot\psi(p)$ and $\CT^{\Gamma_0(Np)}_p\circ[p]^{-}_\psi=\CT^{\Gamma_0(N)}_p-\gamma_p-\psi^{-1}(p)$;

(3) It $p\mid N$, then $\CT^{\Gamma_0(Np)}_p\circ[p]^{+}_\psi=\CT^{\Gamma_0(N)}_p-p\cdot\psi(p)$ and $\CT^{\Gamma_0(Np)}_p\circ[p]^{-}_\psi=\CT^{\Gamma_0(N)}_p-\psi^{-1}(p)$.
\end{lem}
\begin{proof}
Since $\gamma_p$ maps $M_2(\Gamma_0(N),\BC)$ to $M_2(\Gamma_0(Np),\BC)$ and $[p]^{\pm}_\psi$ is defined to be a linear combination of the identity map and $\gamma_p$, we find that $[p]^{\pm}_\psi$ also maps $M_2(\Gamma_0(N),\BC)$ to $M_2(\Gamma_0(Np),\BC)$. Moreover, if $\ell$ is a prime and $\ell\neq p$, then $\gamma_p$ commutes with $\CT_\ell=\sum^{\ell-1}_{k=0}\left(
                                                                                           \begin{array}{cc}
                                                                                             1 & k \\
                                                                                             0 & \ell \\
                                                                                           \end{array}
                                                                                         \right)+\left(
                                                                                                   \begin{array}{cc}
                                                                                                     \ell & 0 \\
                                                                                                     0 & 1 \\
                                                                                                   \end{array}
                                                                                                 \right)
$ (or $\sum^{\ell-1}_{k=0}\left(
         \begin{array}{cc}
           1 & k \\
           0 & \ell \\
         \end{array}
       \right)
$) if $\ell\nmid N$ (or respectively $\ell\mid N$) as operators on corresponding space of modular forms, so the first assertion follows.

If $p\nmid N$, then we have by definition that
\begin{align*}
  \CT^{\Gamma_0(Np)}_p\circ[p]^{+}_\psi(g)&=g|\left[1-\psi(p)\cdot\left(
                                                                    \begin{array}{cc}
                                                                      p & 0 \\
                                                                      0 & 1 \\
                                                                    \end{array}
                                                                  \right)
  \right]|\sum^{p-1}_{k=0}\left(
                               \begin{array}{cc}
                                 1 & k \\
                                 0 & p \\
                               \end{array}
                             \right)\\
  &=g|\sum^{p-1}_{k=0}\left(
                               \begin{array}{cc}
                                 1 & k \\
                                 0 & p \\
                               \end{array}
                             \right)-\psi(p)\cdot g|\sum^{p-1}_{k=0}\left(
                               \begin{array}{cc}
                                 p & pk \\
                                 0 & p \\
                               \end{array}
                             \right)\\
  &=\CT^{\Gamma_0(N)}_p(g)-f|\gamma_p-p\cdot\psi(p)\cdot g,
\end{align*}
for any $g\in M_2(\Gamma_0(N),\BC)$; similarly, we have by definition that
\begin{align*}
  \CT^{\Gamma_0(Np)}_p\circ[p]^{-}_\psi(g)&=g|\left[1-p^{-1}\cdot\psi^{-1}(p)\cdot\left(
                                                                    \begin{array}{cc}
                                                                      p & 0 \\
                                                                      0 & 1 \\
                                                                    \end{array}
                                                                  \right)
  \right]|\sum^{p-1}_{k=0}\left(
                               \begin{array}{cc}
                                 1 & k \\
                                 0 & p \\
                               \end{array}
                             \right)\\
  &=g|\sum^{p-1}_{k=0}\left(
                               \begin{array}{cc}
                                 1 & k \\
                                 0 & p \\
                               \end{array}
                             \right)-p^{-1}\cdot\psi^{-1}(p)\cdot g|\sum^{p-1}_{k=0}\left(
                               \begin{array}{cc}
                                 p & pk \\
                                 0 & p \\
                               \end{array}
                             \right)\\
  &=\CT^{\Gamma_0(N)}_p(g)-f|\gamma_p-\psi^{-1}(p)\cdot g
\end{align*}
so the second assertion follows. The proof of the third assertion is similar and we leave it to the reader.
\end{proof}

3.2. It is well known that the number of cusps of $X_0(DC)$ is equal to $\sum_{1\leq d\mid DC}\varphi(d,{DC}/{d})$, so we find that $dim_{\BC}\CE_2(\Gamma_0(DC),\BC)=\sum_{1< d\mid DC}\varphi(d,{DC}/{d})$. Here $\varphi(d,{DC}/{d})$ means applying Euler's $\varphi$-function to the greatest common divisor of $d$ and $DC/d$. We define $\CH(DC)$ to be the set of all triples $(M,L,\psi)$ where $1\leq M,L\mid D$ with $M\neq1$, $D\mid ML\mid DC$ and $\psi$ is a Dirichlet character modulo $(M,L)$. Note that the condition "$M\neq1$" is automatically satisfied if $\psi\neq1$.

\begin{lem}\label{1.2}
$\#\CH(DC)=dim_{\BC}\ \CE_2(\Gamma_0(DC),\BC)$
\end{lem}
\begin{proof}
By the above remark, we only need to prove that $\#\CH(DC)=\sum_{1< d\mid DC}\varphi(d,\frac{DC}{d})$. We will first prove this when $C=D$. For any positive divisor $d$ of $D^2$, we can associate the following two positive integers
\[M:=\sqrt{d\cdot(d,\frac{D^2}{d})},\ L:=\sqrt{\frac{D^2}{d}\cdot(d,\frac{D^2}{d})}\]
such that $1\leq M,L\mid D$ and $D\mid ML\mid D^2$. Conversely, to any pair of integers $M$ and $L$ with $1\leq M,L\mid D$ and $D\mid ML\mid D^2$, we can associate a positive divisor $d$ of $D$ as
\[d:=\left[\frac{M}{(M,L)}\right]^2\cdot(M,L)\]
It is easy to see that the above establishes a bijection between $\{d:1\leq d\mid D^2\}$ and the set of all pair of integers $M$ and $L$ with $1\leq M,L\mid D$ and $D\mid ML\mid D^2$. Moreover, under this bijection, the divisor $1$ of $D^2$ corresponds to the pair $M=1$ and $L=D$, and we have $(d,{D^2}/{d})=(M,L)$ if $d$ corresponds to $M$ and $L$. It follows that there is a bijection between $\{(d,\psi)|1<d\mid D^2,\psi:({\BZ}/{(d,{D^2}/{d})\cdot\BZ})^\times\rightarrow\BC^\times\}$ and $\CH(D^2)$ which proves the lemma in this situation.

In general, since $DC=\frac{D}{C}\cdot C^2$, any positive divisor $d$ of $DC$ can be uniquely decomposed as $d=d_0\cdot d'$ with $1\leq d_0\mid\frac{D}{C}$ and $1\leq d'\mid C^2$. If such a positive divisor $d'$ of $C^2$ corresponds to a pair of integer $m$ and $\ell$ with $1\leq m,\ell\mid C$ and $C\mid m\ell\mid C^2$ as above, then we can associate with $d$ the pair of integers $M=d_0\cdot m$ and $\frac{DC}{d_0}\cdot\ell$ which satisfies $1\leq M,L\mid D$ and $D\mid ML\mid DC$. This establishes a bijection between $\{d:1\leq d\mid DC\}$ and the set of all pair of integers $M$ and $L$ with $1\leq M,L\mid D$ and $D\mid ML\mid DC$. Moreover, we have $1\mid D^2$ corresponds to the pair $M=1$ and $L=D$, and $(d,\frac{DC}{d})=(M,L)$ if $d$ corresponds to $M$ and $L$. It follows that there is a bijection between $\{(d,\psi):1<d\mid D^2,\psi:\left(\BZ/(d,{DC}/{d})\cdot\BZ\right)^\times\rightarrow\BC^\times\}$ and $\CH(DC)$ which completes the proof the lemma.
\end{proof}

\begin{defn}\label{key definition}
For any Dirichlet character $\psi$ of conductor $f_\psi=f$, let
\begin{align*}
  E_\psi:=-\frac{1}{2g(\psi)}\sum_{a\in(\BZ/f\BZ)^\times}\sum_{b\in(\BZ/f^2\BZ)^\times}\psi(a)\cdot\psi(b)\cdot\phi_{(\frac{a}{f},\frac{b}{f^2})}.
\end{align*}
Then we define
\[E_{M,L,\psi}:=[\frac{L}{f}]^-_\psi\circ[\frac{M}{f}]^+_\psi(E_\psi),\]
for any $(M,L,\psi)\in\CH(DC)$, where $g(\psi)$ is the Gauss sum of $\psi$.
\end{defn}

From Eq.(2.1), it is easy to see that
\begin{align*}
E_\psi&=-\frac{\delta_\psi}{4\pi i(z-\overline{z})}-\frac{1}{4g(\psi)}\sum_{x\in({\BZ}/{f\BZ})^\times}\sum_{y\in({\BZ}/{f^2\BZ})^\times}\psi(x)\cdot\psi(y)\cdot B_2(\frac{x}{f})\\
&+\frac{1}{g(\psi)}\sum_{x\in({\BZ}/{f\BZ})^\times}\sum_{y\in({\BZ}/{f^2\BZ})^\times}\psi(x)\cdot\psi(y)\cdot P_{(\frac{x}{f},\frac{y}{f^2})},
\end{align*}
where $\delta_\psi$ is equal to $1$ or $0$ according to $\psi$ is trivial or not. Since we have by Eq.(2.2) that
\begin{align*}
\sum_{x\in({\BZ}/{f\BZ})^\times}\sum_{y\in({\BZ}/{f^2\BZ})^\times}\psi(x)\cdot\psi(y)\cdot P_{(\frac{x}{f},\frac{y}{f^2})}&=\sum^{\infty}_{k,m=1}\frac{k\psi(k)}{f}\left(\sum_{y\in({\BZ}/{f^2\BZ})^{\times}}\psi(y)e^{2\pi i\frac{my}{f^2}}\right)e^{2\pi i\frac{mk}{f}z}\\
&=\sum^{\infty}_{k,m=1}\frac{k\psi(k)}{f}\left(\sum_{y\in({\BZ}/{f^2\BZ})^{\times}}\psi(y)e^{2\pi i\frac{my}{f}}\right)e^{2\pi imkz}\\
&=g(\psi)\sum^{\infty}_{k,m=1}k\cdot\psi(k)\cdot\psi^{-1}(m)\cdot e^{2\pi imkz},
\end{align*}
with $\psi(n)$ defined to be $0$ when $(n,f)\neq1$ as usual, we find thus
\begin{align}
  E_\psi=-\frac{\delta_\psi}{4\pi i(z-\overline{z})}+a_0(E_\psi;[\infty])+\sum^{\infty}_{n=1}\sigma_{\psi}(n)\cdot \fq^n,
\end{align}
with
\begin{align}
a_0(E_\psi;[\infty])=
\begin{cases}
-\frac{1}{24} &,\text{if}\ \psi=1\\
\ \ 0 &,\text{otherwise}\
\end{cases}
\end{align}
and
\begin{align}
\sigma_\psi(n):=
\sum_{1\leq d\mid n}d\cdot\psi(d)\cdot\psi^{-1}({n}/{d})
\end{align}
In particular, we find that $a_1(E_\psi;[\infty])=1$, which means $E_\psi$ is normalized. Because $[M]^+(\frac{1}{z-\overline{z}})=0$ for any $M>1$ as we have see in Remark~\ref{zero}, it follows from the definition and Eq.(3.1) that $E_{M,L,\psi}$ is always holomorphic and hence belongs to $\CE_2(\Gamma_0(DC),\BC)$.

\begin{lem}\label{normalized}
$E_{M,L,\psi}$ is normalized for any $(M,L,\psi)\in\CH(DC)$.
\end{lem}
\begin{proof}
Because $g|\gamma_p=\sum^{\infty}_{n=0}(p a_n)\cdot {\fq}^{p n}$ for any prime $p$ and function $g$ of the form $\sum^{\infty}_{n=0}a_n\cdot {\fq}^n$, we find that $a_1(g|\gamma_p;[\infty])=0$ and hence $a_1([p]^{\pm}_\psi(g);[\infty])=a_1(g;[\infty])$. By the above discussion, $E_\psi$ is normalized, so the assertion follows.
\end{proof}

\begin{lem}\label{lem2}
For any non-trivial Dirichlet character $\psi$ of conductor $f_\psi=f$, we have that
\[\CT^{\Gamma_0(f^2)}_{\ell}(E_{\psi})=
\begin{cases}
\left(\psi^{-1}(\ell)+\ell\cdot\psi(\ell)\right)\cdot E_\psi &,\text{if}\ \ell\nmid f\\
0 &,\text{if}\ \ell\mid f.\\
\end{cases}
\]
\end{lem}
\begin{proof}
By Proposition 2.4.7 of \cite{St}, we have that
\[\CT^{\Gamma(f^2)}_\ell\left(\phi_{(\frac{x}{f},\frac{y}{f^2})}\right)=\phi_{(\frac{x}{f},\frac{\ell y}{f^2})}+\ell\cdot\phi_{(\frac{\ell'x}{f},\frac{ y}{f^2})}\]
for any prime $\ell\nmid f$, where $\ell'$ is an integer such that $\ell\ell'\equiv1\pmod{f}$ and $\CT^{\Gamma(f^2)}_\ell$ is the $\ell$-th Hecke operator of level $\Gamma(f^2)$. It follows that
\[\CT^{\Gamma_0(f^2)}_{\ell}(E_{\psi})=\left(\psi^{-1}(\ell)+\ell\cdot\psi(\ell)\right)\cdot E_\psi,\]
for any prime $\ell\nmid f$. On the other hand, since
\[E_\psi=-\frac{1}{2g(\psi)}\sum_{x,y\in({\BZ}/{f\BZ})^\times}\psi(x)\cdot\psi(y)\cdot\phi_{(\frac{x}{f},\frac{y}{f})}|\left(
                                                                                                                             \begin{array}{cc}
                                                                                                                               f & 0 \\
                                                                                                                               0 & 1 \\
                                                                                                                             \end{array}
                                                                                                                           \right)
\]
by the distribution law, we find that
\begin{align*}
  \CT^{\Gamma_0(f^2)}_\ell(E_\psi)&=-\frac{1}{2g(\psi)}\sum_{x,y\in({\BZ}/{f\BZ})^\times}\psi(x)\cdot\psi(y)\cdot\phi_{(\frac{x}{f},\frac{y}{f})}|\left(
                                                                                                                             \begin{array}{cc}
                                                                                                                               f & 0 \\
                                                                                                                               0 & 1 \\
                                                                                                                             \end{array}
                                                                                                                           \right)\sum^{\ell-1}_{k=0}\left(
                                                                                                                                                       \begin{array}{cc}
                                                                                                                                                         1 & k \\
                                                                                                                                                          0& \ell \\
                                                                                                                                                       \end{array}
                                                                                                                                                     \right)\\
  &=-\frac{1}{2g(\psi)}\sum_{x,y\in({\BZ}/{f\BZ})^\times}\psi(x)\cdot\psi(y)\cdot\phi_{(\frac{x}{f},\frac{y}{f})}|\sum^{\ell-1}_{k=0}\left(
                                                                                                                                                       \begin{array}{cc}
                                                                                                                                                         1 & \frac{f}{\ell}k \\
                                                                                                                                                          0& 1 \\
                                                                                                                                                       \end{array}
                                                                                                                                                     \right)\left(
                                                                                                                             \begin{array}{cc}
                                                                                                                               f & 0 \\
                                                                                                                               0 & \ell \\
                                                                                                                             \end{array}
                                                                                                                           \right)\\
  &=-\frac{1}{2g(\psi)}\sum_{x,y\in({\BZ}/{f\BZ})^\times}\psi(x)\cdot\psi(y)\sum^{\ell}_{k=0}\phi_{(\frac{x}{f},\frac{y}{f}+\frac{xk}{\ell})}|\left(
                                                                                                                                                     \begin{array}{cc}
                                                                                                                                                       f & 0 \\
                                                                                                                                                       0 & \ell \\
                                                                                                                                                     \end{array}
                                                                                                                                                   \right)=0,
\end{align*}
for any prime $\ell\mid f$, with the last equality holds because $\psi$ is primitive of conductor $f$, and hence complete the proof of the lemma.
\end{proof}

\begin{prop}\label{eigen}
Notations are as above, then we have that
\begin{enumerate}
  \item $E_{M,L,\psi}$ is normalized for any $(M,L,\psi)\in\CH(DC)$, that is to say, $a_1(E_{M,L,\psi};[\infty])=1$ for any $(M,L,\psi)\in\CH(DC)$. In particular, all these Eisenstein series are non-zero;
  \item For any $(M,L,\psi)\in\CH(DC)$, the Hecke operators act on $E_{M,L,\psi}$ as
\[\CT^{\Gamma_0(DC)}_{\ell}(E_{M,L,\psi})=
\begin{cases}
\left(\psi^{-1}(\ell)+\ell\cdot\psi(\ell)\right)\cdot E_{M,L,\psi} &,\text{if}\ \ell\nmid D\\
\psi^{-1}(\ell)\cdot E_{M,L,\psi} &,\text{if}\ \ell\mid \frac{M}{(M,L)}\\
\ell\cdot\psi(\ell)\cdot E_{M,L,\psi} &,\text{if}\ \ell\mid\frac{L}{(M,L)}\\
0 &,\text{if}\ \ell\mid(M,L)
\end{cases}
\]
  \item $\CE_2(\Gamma_0(DC),\BC)=\bigoplus_{(M,L,\psi)\in\CH(DC)}\BC\cdot E_{M,L,\psi}$.
\end{enumerate}
\end{prop}
\begin{proof}
We have already proved the first assertion in Proposition~\ref{normalized}. Lemma~\ref{1.2} implies that the number of the Eisenstein series that we introduced equals the dimension of the $\BC$-vector space $\CE_2(\Gamma_0(DC),\BC)$. Thus, to prove the third assertion, it is enough to show that all these Eisenstein series are linearly independent over $\BC$. So we only need to prove the second assertion, which implies that the Eisenstein series have different eigenvalues and hence are linearly independent.

If $\ell$ is a prime not dividing $D$, then we find by (1) of Lemma~\ref{1.1} and Lemma~\ref{lem2} that
\begin{align*}
  \CT^{\Gamma_0(DC)}_\ell&=[\frac{L}{f}]^-_\psi\circ[\frac{M}{f}]^+_\psi\circ\CT^{\Gamma_0(f^2)}_\ell(E_\psi)\\
  &=\left(\psi^{-1}(\ell)+\ell\cdot\psi(\ell)\right)\cdot E_{M,L,\psi}
\end{align*}

If $\ell$ is a prime divisor of$\frac{M}{(M,L)}$, then we have by (2) of Lemma~\ref{1.1} that
\begin{align*}
 \CT^{\Gamma_0(DC)}_\ell(E_{M,L,\psi})&=[\frac{L}{f}]^-_\psi\circ[\frac{M}{f\ell}]^+_\psi\circ\CT^{\Gamma_0(f^2\ell)}_\ell\circ[\ell]^+_\psi(E_\psi)\\
 &=[\frac{L}{f}]^-_\psi\circ[\frac{M}{f\ell}]^+_\psi\circ(\psi^{-1}(\ell)-\gamma_\ell)(E_\psi)\\
 &=\psi^{-1}(\ell)\cdot E_{M,L,\psi}
\end{align*}

The proofs for those primes $\ell\mid\frac{L}{(M,L)}$ and $\ell\mid\frac{(M,L)}{f}$ are similar to the above, so we omit it here.

Finally, if $\ell\mid f$, then we have that
\[\CT^{\Gamma_0(DC)}_\ell(E_{M,L,\psi})=[\frac{L}{f}]^-_\psi\circ[\frac{M}{f}]^+_\psi\circ\CT^{\Gamma_0(f^2)}_\ell(E_\psi)=0\]
and hence complete the proof.
\end{proof}

\section{The quadratic subgroups of $C_0(DC)$}
\textbf{4.1. }In this section, we study the cuspidal subgroups associated to those $E_{M,L,\psi}$ with $\psi$ a \emph{quadratic character}. We begin with some preliminaries.

\begin{lem}\label{representives}
If we take $r$ to be a positive divisor of $\frac{D}{C}$, and let $s,t$ two positive divisors of $C$ satisfying $(s,t)=1$ and let $x$ runs over a set of representatives of $(\BZ/t\BZ)^\times$ which are prime to $D$, then $\{[\frac{rs^2tx}{DC}]\}$ is a full set of representatives for the cusps of $X_0(DC)$.
\end{lem}
\begin{proof}
It is clear that any divisor of $DC=\frac{D}{C}\cdot C^2$ is of the form $rs^2t$ with some $r,s,t$ as above. Since $(rs^2t,\frac{DC}{rs^2t})=t$ for any such a divisor, we find that the above set has at most $\sum_{1\leq d\mid DC}\varphi(d,\frac{DC}{d})$ elements. Thus, it is enough to prove that the above are all different cusps as the number of cusps of $X_0(DC)$ is also $\sum_{1\leq d\mid DC}\varphi(d,\frac{DC}{d})$.

Suppose $[\frac{r_1s^2_1t_1x_1}{DC}]=[\frac{r_2s^2_2t_2x_2}{DC}]$, then there exists some $\gamma=\left(
                                                                                                    \begin{array}{cc}
                                                                                                      \alpha & \beta \\
                                                                                                      DC\delta & \omega \\
                                                                                                    \end{array}
                                                                                                  \right)\in\Gamma_0(DC)
$ such that $\gamma(\frac{r_1s^2_1t_1x_1}{DC})=\frac{r_2s^2_2t_2x_2}{DC}$. It follows that
\[{r_2s^2_2t_2x_2}={r_1s^2_1t_1}\cdot\frac{\alpha x_1+\beta\frac{DC}{r_1s^2_1t_1}}{\delta r_1s^2_1t_1x_1+\omega}.\]
But since $\delta r_1s^2_1t_1x_1+\omega$ is a unit at every prime dividing $r_1s_1t_1$, we find that $r_1,s_1,t_1$ divides $r_2,s_2,t_2$ respectively, and hence $r_1=r_2,s_1=s_2$ and $t_1=t_2$ by symmetry. If we choose some $u_i,v_i$ ($i=1,2$) such that $\left(
                                                                                                \begin{array}{cc}
                                                                                                  x_i & u_i \\
                                                                                                  \frac{DC}{rs^2t} & v_i \\
                                                                                                \end{array}
                                                                                              \right)\in SL_2(\BZ)
$, then
\[\gamma\cdot\left(
               \begin{array}{cc}
                 x_1 & u_1 \\
                 \frac{DC}{rs^2t} & v_1 \\
               \end{array}
             \right)(\infty)=\left(
               \begin{array}{cc}
                 x_2 & u_2 \\
                 \frac{DC}{rs^2t} & v_2 \\
               \end{array}
             \right)(\infty),
\]
so that there exists some integer $n$ such that
\[\pm\gamma\cdot\left(
               \begin{array}{cc}
                 x_1 & u_1 \\
                 \frac{DC}{rs^2t} & v_1 \\
               \end{array}
             \right)=\left(
               \begin{array}{cc}
                 x_2 & u_2 \\
                 \frac{DC}{rs^2t} & v_2 \\
               \end{array}
             \right)\left(
                      \begin{array}{cc}
                        1 & n \\
                        0 & 1 \\
                      \end{array}
                    \right),
\]
which implies, after a straight forward calculation, that
\[\frac{DC}{rs^2t}v_1-\frac{DC}{rs^2t}v_2\equiv n\cdot\frac{DC}{rs^2t}\cdot\frac{DC}{rs^2t}\pmod{DC}.\]
Because $t^2\mid DC$, it follows that $v_1\equiv v_2\pmod t$. We find thus $x_1\equiv x_2\pmod t$ which completes the proof of the lemma.
\end{proof}

We will always use the above kind of representatives for cusps in the following investigation.
\begin{lem}\label{change}
Let $p$ be a prime divisor of $D$ and $[\frac{rs^2tx}{DC}]$ be a cusp of $X_0(DC)$ , then we have that:
\begin{enumerate}
  \item If $p\mid r$, then $[\frac{rs^2tx}{DC}]=[\frac{(r/p)s^2tx}{DC/p}]$ in $X_0(DC/p)$;
  \item If $p\mid s$, then $[\frac{rs^2tx}{DC}]=[\frac{r(s/p)^2tx}{DC/p^2}]$ in $X_0(DC/p^2)$;
  \item If $p\mid t$, then $[\frac{rs^2tx}{DC}]=[\frac{r(s/p)^2(t/p)\cdot(px)}{DC/p^2}]$ in $X_0(DC/p^2)$;
  \item If $p\mid \frac{D}{Cr}$, then $[\frac{rs^2tx}{DC}]=[\frac{rs^2t\cdot(px)}{DC/p}]$ in $X_0(DC/p)$;
  \item If $p\mid \frac{C}{st}$, then $[\frac{rs^2tx}{DC}]=[\frac{rs^2t\cdot(p^2x)}{DC/p^2}]$ in $X_0(DC/p^2)$.
\end{enumerate}
\end{lem}
\begin{proof}
The first two assertions are obvious. Since the proofs of last three assertions are similar, we will in the following only give that of $(3)$. If $[\frac{rs^2tx}{DC}]=[\frac{r's'^2t'x'}{DC/p^2}]$ in $X_0(DC/p^2)$, then there exists some $\gamma=\left(
                                                                                                         \begin{array}{cc}
                                                                                                           \alpha & \beta \\
                                                                                                           \frac{DC}{p^2}\delta & \omega \\
                                                                                                         \end{array}
                                                                                                       \right)\in \Gamma_0(\frac{DC}{p^2})
$ sending the former point to the latter one, and we find thus
\[r's'^2t'x'=rs^2(t/p)\cdot\frac{x\alpha+\beta\frac{DC}{rs^2t}}{\delta rs^2(t/p)x+\omega p}.\]
Since $\delta rs^2(t/p)x+\omega p$ is a unit for any prime dividing $rs^2(t/p)$, it follows that $r,s,t/p$ divides $r',s',t'$ respectively. We find thus
\[\frac{r'}{r}\cdot\frac{s'^2}{s^2}\cdot\frac{t'}{t/p}\cdot x'=\frac{x\alpha+\beta\frac{DC}{rs^2t}}{\delta rs^2(t/p)x+\omega p}.\]
If there is some prime $q\mid r's't'$ (so that $q\neq p$ as $p\nmid t'$) but not dividing $rst$, then $x\alpha+\beta\frac{DC}{rs^2t}$ will be a $q$-adic unit. But this contradicts to the above equation, so we have proved the assertion.
\end{proof}
Let $K$ be a positive divisor of $D$ and $1\leq\alpha\mid K$. It is not difficult to deduce from the above lemma that: if $(K,rst)=1$, then
\begin{align}\label{4.1}
[\frac{rs^2t\alpha x}{DC}]=[\frac{rs^2t(\frac{K(K,C)}{\alpha}x)}{DC/K(K,C)}]\in X_0(\frac{DC}{K(K,C)});
\end{align}
and if $K\mid t$, then
\begin{align}\label{4.2}
[\frac{rs^2t\alpha x}{DC}]=\frac{rs^2(\frac{t}{K})(\frac{K}{\alpha}x)}{DC/K^2}]\in X_0(\frac{DC}{K^2})
\end{align}
We leave the verifications to the reader. Finally, we give some general observation about how the constant terms of modular forms behave under the operators $[p]^{\pm}_\psi$. Let $N$ be a positive integer and $g\in M_2(\Gamma_0(N),\BC)$. Let $[\frac{a}{c}]$ be a cusp represented by two co-prime integers $a,c$, and let $\gamma=\left(
                                             \begin{array}{cc}
                                               a & b \\
                                               c & d \\
                                             \end{array}
                                           \right)$
be a matrix in $\SL_2(\BZ)$ such that $\gamma([\infty])=[\frac{a}{c}]$. For any prime $p$, we may and will always assume $p\mid d$ when $p\nmid c$. If $p$ is prime to the conductor of $\psi$, then since
\[\gamma_p\cdot\gamma=
\begin{cases}
\left(
           \begin{array}{cc}
             a & pb \\
             c/p & d \\
           \end{array}
         \right)\left(
                  \begin{array}{cc}
                    p & 0 \\
                    0 & 1 \\
                  \end{array}
                \right)
 &,\text{if}\ p\mid c\\
\left(
           \begin{array}{cc}
             ap & b \\
             c & d/p \\
           \end{array}
         \right)\left(
                  \begin{array}{cc}
                    1 & 0 \\
                    0 & p \\
                  \end{array}
                \right)
&,\text{if}\ p\nmid c,
\end{cases}
\]
it follows that
\[a_0([p]^+_\psi(g);[\frac{a}{c}])=
\begin{cases}
a_0(g;[\frac{a}{c}])-p\cdot\psi(p)\cdot a_0(g;[\frac{ap}{c}])
 &,\text{if}\ p\mid c\\
a_0(g;[\frac{a}{c}])-p^{-1}\cdot\psi(p)\cdot a_0(g;[\frac{ap}{c}])
&,\text{if}\ p\nmid c,
\end{cases}
\]
and
\[a_0([p]^-_\psi(g);[\frac{a}{c}])=
\begin{cases}
a_0(g;[\frac{a}{c}])-\psi^{-1}(p)\cdot a_0(g;[\frac{ap}{c}])
 &,\text{if}\ p\mid c\\
a_0(g;[\frac{a}{c}])-p^{-2}\cdot\psi^{-1}(p)\cdot a_0(g;[\frac{ap}{c}])
&,\text{if}\ p\nmid c.
\end{cases}
\]
Thus, for any positive square-free integer $K$ prime to the conductor of $\psi$, we find by induction that
\begin{align}
  a_0([K]^+_\psi(g);[\frac{a}{c}])=
\begin{cases}
\sum_{1\leq\alpha\mid K}(-1)^{\nu(\alpha)}\cdot\alpha\cdot\psi(\alpha)\cdot a_0(g;[\frac{\alpha a}{c}])
 &,\text{if}\ K\mid c\\
\sum_{1\leq\alpha\mid K}(-1)^{\nu(\alpha)}\cdot\alpha^{-1}\cdot\psi(\alpha)\cdot a_0(g;[\frac{\alpha a}{c}])
&,\text{if}\ (K,c)=1,
\end{cases}
\end{align}
and
\begin{align}
  a_0([K]^-_\psi(g);[\frac{a}{c}])=
\begin{cases}
\sum_{1\leq\alpha\mid K}(-1)^{\nu(\alpha)}\cdot\psi^{-1}(\alpha)\cdot a_0(g;[\frac{\alpha a}{c}])
 &,\text{if}\ K\mid c\\
\sum_{1\leq\alpha\mid K}(-1)^{\nu(\alpha)}\cdot\alpha^{-2}\cdot\psi^{-1}(\alpha)\cdot a_0(g;[\frac{\alpha a}{c}])
&,\text{if}\ (K,c)=1.
\end{cases}
\end{align}

\textbf{4.2. The constant terms of $E_{M,L,\psi}$.} Let $\psi$ be a Dirichlet character of conductor $f_\psi=f$. We extend $\psi$ to a function on $\BZ$ so that $\psi(n)=0$ if $(n,f)\neq1$. For any cusp $[\frac{s^2tx}{f^2}]\in X_0(f^2)$ with $s,t\mid f$ and $(s,t)=1$ as in Lemma~\ref{representives}, we can choose a matrix $\left(
                                                                                                                     \begin{array}{cc}
                                                                                                                       x & u \\
                                                                                                                       \frac{f^2}{s^2t} & v \\
                                                                                                                     \end{array}
                                                                                                                   \right)
\in \SL_2(\BZ)$ which maps $[\infty]$ to $[\frac{s^2tx}{f^2}]$. Then it follows from Eqs. (2.1) and (2.4) that
\begin{align*}
  a_0(E_\psi;[\frac{s^2tx}{f^2}])&=-\frac{1}{4g(\psi)}\sum_{a\in({\BZ}/{f\BZ})^\times}\sum_{b\in({\BZ}/{f^2\BZ})^\times}\psi(a)\cdot\psi(b)\cdot B_2(\frac{xa}{f}+\frac{b}{s^2t})\\
  &=-\frac{1}{4g(\psi)}\sum_{b\in({\BZ}/{f^2\BZ})^\times}\psi(b)\left(\sum_{a\in({\BZ}/{f\BZ})^\times}\psi(a)\cdot B_2(\frac{xa}{f}+\frac{b}{s^2t})\right).
\end{align*}
Since the function in the above bracket depends only on $b\pmod{s^2t}$ and $\psi$ is primitive of conductor $f$, we find that $a_0(E_\psi;[\frac{s^2tx}{f^2}])$ must be zero unless $st=f$. However, if $st=f$, then
\begin{align*}
  a_0(E_\psi;[\frac{s^2tx}{f^2}])&=-\frac{1}{4g(\psi)}\sum_{a\in({\BZ}/{f\BZ})^\times}\sum_{b\in({\BZ}/{f^2\BZ})^\times}\psi(a)\cdot\psi(b)\cdot B_2(\frac{xa}{f}+\frac{b}{sf})\\
  &=-\frac{\psi^{-1}(x)}{4g(\psi)}\sum_{a\in({\BZ}/{f\BZ})^\times}\psi(a)\left(\sum_{b,k\in({\BZ}/{f\BZ})^\times}\psi(b)\cdot B_2(\frac{as+b+kf}{fs})\right),
\end{align*}
with the function in the bracket depends only on $a\pmod{\frac{f}{s}}$ and hence is zero unless $s=1$. It follows that
\begin{align}
a_0(E_\psi;[\frac{s^2tx}{f^2}])=
\begin{cases}
\psi^{-1}(x)\cdot n_\psi&,\text{if}\ s=1\ and\ t=f\\
0 &,\text{otherwise,}\
\end{cases}
\end{align}
where
\[n_\psi:=-\frac{f}{4g(\psi)}\sum_{a,b\in{\BZ}/{f\BZ}}\psi(a)\cdot\psi(b)\cdot B_2(\frac{a+b}{f}).\]
In particular, we find that
\begin{align}
  a_0(E_\psi;[\frac{s^2t(\alpha x)}{f^2}])=\psi^{-1}(\alpha)\cdot a_0(E_\psi;[\frac{s^2tx}{f^2}]),
\end{align}
where $\alpha$ is any integer prime to $f$. While the above is valid for any $\psi$ (not necessarily quadratic), we will assume $\psi$ is quadratic in the rest of this paper.

\begin{lem}\label{constant1}
For any quadratic character $\psi$ of conductor $f\mid C$, the constant terms of $E_{D,f,\psi}$ are given as
\[
a_0(E_{D,f,\psi};[\frac{rs^2tx}{DC}])=
\begin{cases}
\varphi(\frac{D}{f})\cdot n_\psi\cdot\frac{(-1)^{\nu(\frac{D}{frs})}\psi(\frac{DC}{frs^2tx})}{rs}&,\text{if}\ (s,f)=1\ and\ f\mid t\\
0 &,\text{otherwise.}\
\end{cases}
\]
In particular, $a_0(E_{D,f,\psi};[\frac{rs^2t(\alpha x)}{DC}])=\psi(\alpha)\cdot a_0(E_{D,f,\psi};[\frac{rs^2tx}{DC}]))$ for any integer $\alpha$ prime to $D$.
\end{lem}
\begin{proof}

Recall that $E_{D,f,\psi}$ is defined as $[\frac{D}{f}]^+_\psi(E_\psi)$. For any cusp $[\frac{rs^2tx}{DC}]$ of $X_0(DC)$, we decompose $\frac{D}{f}$ as $\frac{D}{f}=K_r\cdot K_s\cdot K_t\cdot K$ with $K_r:=(\frac{D}{f},r)=r,\ K_s:=(\frac{D}{f},s)$ and $K_t:=(\frac{D}{f},t)$. By Eqs. (4.1), (4.2) and the first formula of Eq. (4.3), we find that
\begin{align*}
  &a_0(E_{D,f,\psi};[\frac{rs^2tx}{DC}])\\
  &=\sum_{1\leq\alpha\mid K}(-1)^{\nu(\alpha)}\cdot\psi(\alpha)\cdot\alpha\cdot a_0(E_{\frac{D}{K},f,\psi};[\frac{rs^2t(\frac{K(K,C)}{\alpha}x)}{DC/K(K,C)}])\\
  &=\sum_{1\leq\alpha\mid K,1\leq\alpha_t\mid K_t}(-1)^{\nu(\alpha\alpha_t)}\cdot\psi(\alpha\alpha_t)\cdot\alpha\alpha_t\cdot  a_0(E_{\frac{D}{K_tK},f,\psi};[\frac{rs^2(\frac{t}{K_t})(\frac{K_tK(K,C)}{\alpha_t\alpha}x)}{DC/K^2_tK(K,C)}]).
\end{align*}
It then follows from the second formula of Eq. (4.3) together with (1) and (2) of Lemma~\ref{change} that
\begin{align*}
  &a_0(E_{D,f,\psi};[\frac{rs^2tx}{DC}])\\
  &=\sum(-1)^{\nu(\alpha_r\alpha_s\alpha_t\alpha)}\cdot\psi(\alpha_r\alpha_s\alpha_t\alpha)\cdot(\alpha_r\alpha_s)^{-1}\cdot(\alpha_t\alpha)\cdot  a_0(E_\psi;[\frac{(\frac{s}{K_s})^2(\frac{t}{K_t})(\frac{K_tK(K,C)}{\alpha_t\alpha}\alpha_r\alpha_s x)}{f^2}])\\
  &=\psi(K_tK(K,C))\cdot\sum(-1)^{\nu(\alpha_r\alpha_s\alpha_t\alpha)}\cdot(\alpha_r\alpha_s)^{-1}\cdot(\alpha_t\alpha)\cdot  a_0(E_\psi;[\frac{(\frac{s}{K_s})^2(\frac{t}{K_t})x}{f^2}]),
\end{align*}
where $\alpha_r,\alpha_s,\alpha_t$ and $\alpha$ runs through all the positive divisors of $K_r,K_s,K_t$ and $K$ respectively. It follows from (4.5) and (4.6) that the above constant term equals
\[\psi(K_tK(K,C))\cdot\prod_{p\mid K_rK_s}(1-\frac{1}{p})\cdot\prod_{p\mid K_tK}(1-p)\cdot a_0(E_\psi;[\frac{(\frac{s}{K_s})^2(\frac{t}{K_t})x}{f^2}]),\]
which is zero unless $s=K_s$ and $fK_t\mid t$, or equivalently, $(s,f)=1$ and $f\mid t$. Moreover, if these conditions are satisfied, then $K_rK_s=rs, K_tK=\frac{D}{frs}$ and $(K,C)=\frac{C}{st}$, which completes the proof.
\end{proof}

\begin{lem}\label{constant2}
For any quadratic character $\psi$, the constant terms of $E_{M,f\frac{D}{M},\psi}$ are given as
\[a_0(E_{M,f\frac{D}{M},\psi};[\frac{rs^2tx}{DC}])=
\begin{cases}
\varphi(\frac{D}{f})\cdot\mu(\frac{D}{M})\cdot n_\psi\cdot\frac{(-1)^{\nu(\frac{D}{frs})}\psi(\frac{DC}{frs^2tx})}{rs\frac{D}{M}}
 &,\text{if}\ \frac{D}{M}\mid rs,\ (s,f)=1\ and\ f\mid t\\
0
&,\text{otherwise,}\
\end{cases}\]
where $\mu(n)=\prod_{p\mid n}(1+p)$ for any positive integer $n$. In particular, $a_0(E_{M,f\frac{D}{M},\psi};[\frac{rs^2t(\alpha x)}{DC}])=\psi(\alpha)\cdot a_0(E_{M,f\frac{D}{M},\psi};[\frac{rs^2tx}{DC}]))$ for any integer $\alpha$ prime to $D$.
\end{lem}
\begin{proof}
Recall that, for any $M$ divided by $f$, $E_{M,f\cdot\frac{D}{M},\psi}$ is defined as $[\frac{D}{M}]^-_\psi(E_{M,f,\psi})$. For any cusp $[\frac{rs^2tx}{DC}]$ of $X_0(DC)$, we decompose $\frac{D}{M}$ as $\frac{D}{M}=H_r\cdot H_s\cdot H_t\cdot H$ with $H_r:=(\frac{D}{M},r),\ H_s:=(\frac{D}{M},s)$ and $H_t:=(\frac{D}{M},t)$. By Eqs. (4.1), (4.2) and the first formula of Eq. (4.4), we find that
\begin{align*}
  &a_0(E_{M,f\cdot\frac{D}{M},\psi};[\frac{rs^2tx}{DC}])\\
  &=\sum_{1\leq\alpha\mid H}(-1)^{\nu(\alpha)}\cdot\psi^{-1}(\alpha)\cdot a_0(E_{M,f\cdot\frac{D}{MH},\psi};[\frac{rs^2t(\frac{H(H,C)}{\alpha}x)}{DC/H(H,C)}])\\
   &=\sum_{1\leq\alpha\mid H,1\leq\alpha_t\mid H_t}(-1)^{\nu(\alpha_t\alpha)}\cdot\psi^{-1}(\alpha_t\alpha)\cdot
   a_0(E_{M,f\cdot\frac{D}{MH_tH},\psi};[\frac{rs^2(\frac{t}{H_t})(\frac{H_tH(H,C)}{\alpha_t\alpha}x)}{DC/H^2_tH(H,C)}]).
\end{align*}
It then follows from the second formula of Eq. (4.4), (1) and (2) of Lemma~\ref{change} and the last assertion of Lemma 4.3 that
\begin{align*}
  &a_0(E_{M,f\cdot\frac{D}{M},\psi};[\frac{rs^2tx}{DC}])\\
  &=\sum(-1)^{\nu(\alpha_r\alpha_s\alpha_t\alpha)}\cdot\psi^{-1}(\alpha_r\alpha_s\alpha_t\alpha)\cdot(\alpha_r\alpha_s)^{-2}\cdot  a_0(E_{M,f,\psi};[\frac{(\frac{r}{H_r})(\frac{s}{H_s})^2(\frac{t}{H_t})(\frac{H_tH(H,C)}{\alpha_t\alpha}\alpha_r\alpha_s x)}{M\cdot(M,C)}])\\
  &=\psi(HH_t(H,C))\sum(-1)^{\nu(\alpha_r\alpha_s\alpha_t\alpha)}\cdot(\alpha_r\alpha_s)^{-2}\cdot  a_0(E_{M,f,\psi};[\frac{(\frac{r}{H_r})(\frac{s}{H_s})^2(\frac{t}{H_t})x}{M\cdot(M,C)}]),
\end{align*}
where $\alpha_r,\alpha_s,\alpha_t$ and $\alpha$ runs through all the positive divisors of $H_r,H_s,H_t$ and $H$ respectively. It is easy to see that the above sum is zero unless $H_t=H=1,(s,f)=1$ and $f\mid t$, or equivalently, $\frac{D}{M}\mid rs,(s,f)=1$ and $f\mid t$. When these conditions are satisfied, then the assertion follows from the previous Lemma.
\end{proof}

\begin{prop}\label{constant3}
For any $(M,L,\psi)\in\CH(DC)$ with $\psi$ a quadratic character of conductor $f_\psi=f$, the constant terms of $E_{M,L,\psi}$ are given as
\begin{align*}
a_0(E_{M,L,\psi};[\frac{rs^2tx}{DC}])=
\begin{cases}
n_\psi\cdot\frac{\varphi(\frac{D}{f})\cdot\mu(\frac{L}{f})}{L/f}\cdot c_{rstx}&,\text{if}\ (s,f)=1, (M,L)\mid st\ and\ \frac{D}{M}\mid rs\\
0 &,\text{otherwise,}
\end{cases}
\end{align*}
where $c_{rstx}:=\frac{(-1)^{\nu(\frac{D}{frs})}\psi(\frac{DC}{frs^2tx})}{rs}\prod_{p\mid(s,\frac{(M,L)}{f})}(1-\frac{1}{p})$.
\end{prop}
\begin{proof}
We have already proved the assertion when $(M,L)=f$, so it remains to consider the case when $(M,L)\neq f$. Since $(M,L)\mid C$, $\frac{(M,L)}{f}$ can be decomposed as $\frac{(M,L)}{f}=W_s\cdot W_t\cdot W$ for any cusp $[\frac{rs^2tx}{DC}]$ of $X_0(DC)$, where $W_s:=(\frac{(M,L)}{f},s)$ and $W_t:=(\frac{(M,L)}{f},t)$. It then follows from Eq. (4.4) that
\begin{align*}
  a_0(E_{M,L,\psi};[\frac{rs^2tx}{DC}])&=\sum(-1)^{\nu(\alpha)}\cdot\psi(\alpha)\cdot a_0(E_{M,f\cdot\frac{D}{M}\cdot W_s\cdot W_t,\psi};[\frac{rs^2t\alpha x}{DC}])\\
  &=\sum(-1)^{\nu(\alpha\alpha_t)}\cdot\psi(\alpha\alpha_t)\cdot a_0(E_{M,f\cdot\frac{D}{M}\cdot W_s,\psi};[\frac{rs^2t\alpha\alpha_t x}{DC}])\\
  &=\sum(-1)^{\nu(\alpha\alpha_t\alpha_s)}\cdot\psi(\alpha\alpha_t\alpha_s)\cdot\alpha^{-2}_s a_0(E_{M,f\cdot\frac{D}{M},\psi};[\frac{rs^2t\alpha\alpha_t\alpha_s x}{DC}]),
\end{align*}
where $\alpha_s,\alpha_t$ and $\alpha$ runs over all positive divisors of $W_s,W_t$ and $W$ respectively. As a cusp of $X_0(DC)$, we have
\begin{align*}
[\frac{rs^2t\alpha\alpha_t\alpha_s x}{DC}]=[\frac{r(s\alpha_t)^2(\frac{t\alpha}{\alpha_t})(\alpha_s x+\frac{DC}{\alpha^2_s})}{DC}]
\end{align*}
with $(\alpha_s x+\frac{DC}{\alpha^2_s},D)=1$, and $\alpha_s x+\frac{DC}{\alpha^2_s}\equiv \alpha_sx\pmod{f}$ because $(\alpha_s,f)=1$. So we find by Lemma~\ref{constant2} that
\begin{align*}
  a_0(E_{M,f\cdot\frac{D}{M},\psi};[\frac{rs^2t\alpha\alpha_t\alpha_s x}{DC}])=(-1)^{\nu(\alpha_t)}\cdot\psi(\alpha\alpha_t\alpha_s)\cdot\alpha_t^{-1}\cdot a_0(E_{M,f\frac{D}{M},\psi};[\frac{rs^2tx}{DC}]),
\end{align*}
and hence
\[a_0(E_{M,L,\psi};[\frac{rs^2tx}{DC}])=\sum(-1)^{\nu(\alpha\alpha_s)}\cdot\alpha^{-1}_t\cdot\alpha^{-2}_s a_0(E_{M,f\cdot\frac{D}{M},\psi};[\frac{rs^2tx}{DC}]).\]
Thus, the constant term is zero unless $\frac{D}{M}\mid rs,(s,f)=1,f\mid t\text{ and }W=1$, or equivalently, $\frac{D}{M}\mid rs,(s,f)=1\text{ and }(M,L)\mid st$. If these conditions are satisfied, then it is easy to derive the desired result from the previous lemma.
\end{proof}

\begin{cor}\label{R}For any quadratic character $\psi$ of conductor $f$, we have that
\[\CR_{\Gamma_0(DC)}(E_{M,L,\psi})=n_\psi\cdot\frac{\varphi(\frac{D}{f})\cdot\mu(\frac{L}{f})\cdot(\frac{D}{M},C)}{L/f}\BZ\]
and
\[\CR_{\Gamma_1(DC)}(E_{M,L,\psi})=n_\psi\cdot\frac{\varphi(\frac{D}{f})\cdot\mu(\frac{L}{f})\cdot(\frac{D}{M},C)\cdot f}{L/f}\BZ.\]
\end{cor}
\begin{proof}
This follows immediately from the above result about constant terms, since the ramification index of $X_0(DC)$ at the cusp $[\frac{rs^2tx}{DC}]$ equals to $rs^2$, and the ramification index of $X_1(DC)$ at a cusp over $[\frac{sr^2tx}{DC}]$ equals to $rs^2t$.
\end{proof}

\textbf{4.3. The periods of $E_{M,L,\psi}$.} Now we turn to the determination of the periods of the Eisenstein series $E_{M,L,\psi}$ with $\psi$ being a quadratic character.
\begin{lem}\label{Fourier'}
For any quadratic character $\psi$ of conductor $f$, the Fourier expansion of $E_{D,f,\psi}$ at $[\infty]$ is given as
\begin{align*}
  E_{D,f,\psi}=a_0(E_{D,f,\psi};[\infty])+\sum^{\infty}_{n=1}\sigma_{\frac{D}{f}}(n)\cdot\psi(n)\cdot \fq^n,
\end{align*}
with $\sigma_{\frac{D}{f}}(n):=\sum_{1\leq d\mid n,(d,\frac{D}{f})=1}d$ for any positive integer $n$.
\end{lem}
\begin{proof}
We prove the statement by induction on $\nu(\frac{D}{f})$. Because $\psi$ is quadratic, it follows from Eqs. (3.1) and (3.3) that $a_n(E_\psi;[\infty])=(\sum_{1\leq d\mid n}d)\cdot\psi(n)$ for any $n\geq1$, which verifies the assertion if $D=f$. Suppose $\frac{D}{f}\neq1$ and let $p$ be an arbitrary prime divisor of it. Because the non-holomorphic terms is annihilated by $[p]^+_\psi$ (see Remark~\ref{zero}), it follows from the induction hypothesis that
\begin{align*}
  E_{D,f,\psi}&=[p]^+_\psi(E_{\frac{D}{p},f,\psi})\\
  &=\left(a_0(E_{\frac{D}{p},f,\psi})+\sum^{\infty}_{n=1}\sigma_{\frac{D}{fp}}(n)\cdot\psi(n)\cdot \fq^n\right)-p\cdot\psi(p)\cdot\left(a_0(E_{\frac{D}{p},f,\psi})+\sum^{\infty}_{n=1}\sigma_{\frac{D}{fp}}(n)\cdot\psi(n)\cdot \fq^{pn}\right)\\
  &=a_0(E_{D,f,\psi})+\sum^{\infty}_{n=1}\left(\sigma_{\frac{D}{pf}}(n)-p\cdot\sigma_{\frac{D}{pf}}(n/p)\right)\cdot\psi(n)\cdot \fq^n,
\end{align*}
with $\frac{n}{p}$ defined to be $0$ if $p\nmid n$. It is easy to see that $\sigma_{\frac{D}{pf}}(n)-p\cdot\sigma_{\frac{D}{pf}}(\frac{n}{p})=\sigma_{\frac{D}{f}}(n)$ for any positive integer $n$ and so we are done.
\end{proof}

\begin{lem}\label{Fourier}
For any quadratic character $\psi$ of conductor $f$, the Fourier expansion of $E_{M,L,\psi}$ at $[\infty]$ is given as
\[E_{M,L,\psi}=a_0(E_{M,L,\psi})+\sum^{\infty}_{n=1}\sigma_{M,L}(n)\cdot\psi(n)\cdot \fq^n,\]
where $\sigma_{M,L}(n)$ is defined to be $(\sum_{1\leq d\mid n,(d,\frac{D}{f})=1}d)\cdot(\prod_{\ell\mid\frac{D}{M}}\ell^{v_\ell(n)})$ or zero according to $n$ is prime to $(M,L)$ or not.
\end{lem}
\begin{proof}
We first consider the case when $(M,L)=f$ so that $E_{M,L,\psi}=E_{M,f\cdot\frac{D}{M},\psi}$. We will prove the lemma in this situation by induction on $\nu(\frac{D}{M})$. If $\frac{D}{M}=1$, then the assertion have already been verified in the previous lemma.
If $\frac{D}{M}>1$ and let $p$ be an arbitrary prime divisor of it. Then it follows form the induction hypothesis that
\begin{align*}
E_{M,f\cdot\frac{D}{M},\psi}&=[p]^-_{\psi}(E_{M,f\cdot\frac{D}{pM},\psi})\\
&=a_0(E_{M,f\cdot\frac{D}{M},\psi};[\infty])+\sum^{\infty}_{n=1}\left(\sigma_{M,f\cdot\frac{D}{pM}}(n)-\sigma_{M,f\cdot\frac{D}{pM}}(n/p)\right)\cdot\psi(n)\cdot \fq^n
\end{align*}
Writing $n$ as $m\cdot p^{v_p(n)}$ with $(m,p)=1$, then we find that
\begin{align*}
  &\sigma_{M,f\frac{D}{pM}}(n)-\sigma_{M,f\frac{D}{pM}}(n/p)\\
  &=(p^{v_p(n)}+...+1)\cdot\sigma_{M,f\cdot\frac{D}{pM}}(m)-(p^{v_p(n)-1}+...+1)\cdot\sigma_{M,f\cdot\frac{D}{pM}}(m)\\
  &=p^{v_p(n)}\cdot\sigma_{M,f\cdot\frac{D}{pM}}(n),
\end{align*}
which proves the assertion in this case. In general, if $(M,L)\neq1$, then we choose an arbitrary prime divisor $p\mid (M,L)\mid C$ and find that
\begin{align*}
  E_{M,L,\psi}&=[p]^-_{\psi}(E_{M,\frac{L}{p},\psi})\\
  &=a_0(E_{M,L,\psi};[\infty])+\sum^{\infty}_{n=1}\left(\sigma_{M,\frac{L}{p}}(n)-\sigma_{M,\frac{L}{p}}(n/p)\right)\cdot\psi(p)\cdot e^{2\pi inz}.
\end{align*}
We have thus complete the proof of the lemma since it is easy to see that $\sigma_{M,\frac{L}{p}}(n)-\sigma_{M,\frac{L}{p}}(n/p)=0$ if $p\mid n$.
\end{proof}

\begin{prop}\label{P}
For any quadratic character $\psi$ of conductor $f$, we have $\CP_{\Gamma_1(DC)}(E_{M,L,\psi})=\frac{g(\psi)}{L}\BZ+\CR_{\Gamma_1(DC)}(E_{M,L,\psi})$.
\end{prop}
\begin{proof}
Straight manipulation with the Fourier expansion of $E_{M,l,\psi}$ given by Lemma~\ref{Fourier} yields that
\begin{align*}
L(E_{M,L,\psi},\chi,s)=\prod_{p\mid{M}/{f}}(1-\chi\psi(p)\cdot p^{1-s})\cdot\prod_{p\mid{L}/{f}}(1-\chi\psi(p)\cdot p^{-s})\cdot L(\chi\psi,s-1)\cdot L(\chi\psi,s),
\end{align*}
for any Dirichlet character $\chi$ of conductor prime to $D$. It follows that $\Lambda(E_{M,L,\psi},\chi,1)=0$ if $\chi\psi(-1)=1$, and
\begin{align*}
  \Lambda(E_{M,L,\psi},\chi,1)=
-\frac{\chi(-f)\psi(f_\chi)g(\psi)}{2f}\cdot\prod_{p\mid{M}/{f}}(1-\chi\psi(p))\cdot\prod_{p\mid{L}/{f}}(1-\frac{\chi\psi(p)}{p})\cdot B_{1,\chi\psi}\cdot B_{1,\overline{\chi\psi}}
\end{align*}
if $\chi\psi(-1)=-1$. By 4.2 (b) of \cite{St2}, this implies that$\frac{g(\psi)}{L}\BZ+\CR_{\Gamma_1(DC)}(E_{M,L,\psi})$ satisfies the condition (St3), and hence $\CP_{\Gamma_1(DC)}(E_{M,L})\subseteq\frac{g(\psi)}{L}\BZ+\CR_{\Gamma_1(DC)}(E_{M,L,\psi})$. Thus, it remains to prove $\CP_{\Gamma_1(DC)}(E_{M,L,\psi})\supseteq\frac{g(\psi)}{L}\BZ$.

Let $q$ be an arbitrary prime. For any prime $p'\in S_{DC}$ not equal to $q$, both $\prod_{p\mid \frac{M}{f}}(\psi(p)-\chi(p))$ and $\prod_{p\mid \frac{L}{f}}(\psi(p)\cdot p-\chi(p))$ are $q$-adic units for all but finitely many $\chi\in\fX^{\infty}_{DC}$ whose conductor is a power of $p'$. It then follows from the above $L$-value formula and Theorem 4.2 (c) of \cite{St2} that $\frac{L}{g(\psi)}\cdot\Lambda(E_{M,L,\psi},\chi,1)$ is a $q$-adic unit for infinitely many $\chi\in\fX^{\infty}_{DC}$ and hence completes the proof.
\end{proof}

\begin{thm}\label{order} Let $\psi$ be a quadratic character of conductor $f$, then
\[C(E_{M,L,\psi})\otimes_{\BZ}\BZ[\frac{1}{2^{\delta_{M,L}}(M,L)}]\simeq\frac{\frac{g(\psi)}{f\cdot n_\psi}\BZ+\varphi(\frac{D}{f})\cdot\mu(\frac{L}{f})\cdot(\frac{D}{M},C)\BZ}{\varphi(\frac{D}{f})\cdot\mu(\frac{L}{f})\cdot(\frac{D}{M},C)\BZ}\otimes_{\BZ}\BZ[\frac{1}{2^{\delta_{M,L}}(M,L)}] \]
where $\delta_{M,L}$ equals $1$ or $0$ according to $(M,L)=1$ or not.
\end{thm}
\begin{proof}
It follows from Corollary~\ref{R} and Proposition~\ref{P} that
\[A^{(s)}(E_{M,L,\psi})=\frac{\CP_{\Gamma_1(DC)}E_{M,L,\psi}+\CR_{\Gamma_0(DC)}E_{M,L,\psi}}{\CR_{\Gamma_0(DC)}E_{M,L,\psi}}\simeq\frac{\frac{g(\psi)}{f\cdot n_\psi}\BZ+\varphi(\frac{D}{f})\cdot\mu(\frac{L}{f})\cdot(\frac{D}{M},C)\BZ}{\varphi(\frac{D}{f})\cdot\mu(\frac{L}{f})\cdot(\frac{D}{M},C)\BZ}\]
Since the intersection $C(E_{M,L,\psi})$ is annihilated by $T_p$ for any $p\mid(M,L)$ and such $T_p$ acts on $\sum_{DC}$ as multiplication by $p$ by \cite{LO}, it follows that $\sum_{DC}\bigcap C(E_{M,L,\psi})$ is annihilated by $(M,L)$ and hence finishes the proof when $(M,L)\neq1$.

However, if $(M,L)=1$ and hence $\psi=1$, then the cyclic group $\sum_{DC}\bigcap C(E_{M,L,\psi})$ is both of multiplicative type and $\BQ$-rational, so it must be contained in $\mu_2$. In particular, $\sum_{DC}\bigcap C(E_{M,L})$ is annihilated by $2$, and the result follows.
\end{proof}

\section{Proof of Theorems~\ref{M1} and \ref{M2}}
\textbf{5.1. The new part of $J_0(N)$.} Let $N$ be a positive integer. For any positive divisors $n\mid N$ and $m\mid \frac{N}{n}$, we have the following homomorphism
\[S_2(\Gamma_0(n),\BC)\rightarrow S_2(\Gamma_0(N),\BC),\]
which maps $f(z)$ to $f(mz)$, and hence the following
\[\prod_{n\mid N, n\neq N, m\mid\frac{N}{n}}S_2(\Gamma_0(n),\BC)\rightarrow S_2(\Gamma_0(N),\BC),\]
whose cokernel is isomorphic to the subspace of new forms of level $\Gamma_0(N)$. The above homomorphism induces the following morphism between abelian varieties over $\BQ$
\[{\iota_N}: J_0(N) \rightarrow\prod_{n\mid N, n\neq N,m\mid\frac{N}{n}}J_0(n).\]
The \emph{new part $J^{new}_0(N)$} of $J_0(N)$ is then defined to be the kernel of the above morphism, so we have the following cartesian diagram
\[\xymatrix{
  J^{new}_0(N) \ar[d] \ar[r]
                & 0 \ar[d]  \\
  J_0(N)  \ar[r]
                & \prod_{n\mid N, n\neq N,m\mid\frac{N}{n}}J_0(n)}.\]

\textbf{5.2. Proof of Theorem~\ref{M1}.} In fact, we \textbf{claim} that $J_0(N)(\BQ)[q^\infty]=0$ for any prime $q\nmid 6\cdot N\cdot\varpi(N)$ which clearly implies Theorem~\ref{M1}. We prove this claim by induction on $\nu(N)$. When $\nu(N)=1$ so that $N$ is a prime, the claim follows from the theorems of Ogg and Mazur. In general, if $q$ is a prime such that $q\mid6\cdot N\cdot\varpi(N)$, then we also have $q\nmid 6\cdot n\cdot\varpi(n)$ for any $n\mid N$. Thus, by the induction hypothesis, a point $P\in J_0(N)(\BQ)[q^\infty]$ must be mapped to zero by $\iota_N$ as $\nu(n)<\nu(N)$ for any $n\mid N$ and $n\neq N$. It follows that $P\in J^{new}_0(N)(\BQ)[q^\infty]$ and we are reduced to prove that $J^{new}_0(N)(\BQ)[q^\infty]=0$ for any prime $q\nmid 6\cdot N\cdot\varpi(N)$.

We can write $N$ as $D\cdot C\cdot C_1\cdot\cdot\cdot C_k$, where $D,C,C_1,...,C_k$ are positive square-free integers such that $C_k\mid C_{k-1}\mid...\mid C\mid D$. By the Eichler-Shimura theory, we have $T^{\Gamma_0(N)}_\ell(P)=(1+\ell)\cdot P$ for any prime $\ell\nmid D$. Moreover, by the newform theory, we have $T^{\Gamma_0(N)}_\ell$ acts on $J^{new}_0(N)$ as multiplication by $\epsilon_\ell$, where $\epsilon_\ell=\pm1$ if $\ell\mid (D/C)$ and $\epsilon_\ell=0$ if $\ell\mid C$.

Thus, if $0\neq P\in J^{new}_0(N)(\BQ)[q^\infty]$, then we have
\[S_2(\Gamma_0(N),\BF_q)\left[\{T^{\Gamma_0(N)}_\ell-(1+\ell)\}_{\ell\nmid D}, \{T^{\Gamma_0(N)}_\ell-\epsilon_\ell\}_{\ell\mid D}\right]\neq0\]
and is generated by a unique normalized $\Theta$. However, simple manipulation shows that

$\bullet$ If $\epsilon_\ell=1$, then $[\ell]^-(\Theta)$ belongs to $S_2(\Gamma_0(N\ell),\BF_q)$ and is annihilated by $T^{\Gamma_0(N\ell)}_\ell$;

$\bullet$ If $\epsilon_\ell=-1$, then $\Theta+\frac{1}{\ell}\Theta|\gamma_\ell$ belongs to $S_2(\Gamma_0(N\ell),\BF_q)$ and is annihilated by $T^{\Gamma_0(N\ell)}_\ell$.

Thus, by raising the levels in such a way, we will finally get some normalized form which spans the following one-dimensional $\BF_q$-vector space
\[S_2(\Gamma_0(ND/C),\BF_q)[\{T^{\Gamma_0(ND/C)}_\ell-(1+\ell)\}_{\ell\nmid D},\{T^{\Gamma_0(ND/C)}_\ell\}_{\ell\mid D}],\]
with $ND/C=D^2\cdot C_1\cdot\cdot\cdot C_k$ being a multiple of $D^2$. By the $\fq$-expansion principle and Proposition~\ref{eigen}, this normalized form is exactly $E_{D,D}$ modulo $q$. In particular, we find that $E_{D,D}$ must be a modulo $q$ cusp form, so that all its constant terms should be zero modulo $q$. But by Proposition~\ref{constant3}, the non-zero constant terms of $E_{D,D}$ are all units in $\BZ[\frac{1}{6\cdot D\cdot \varpi(D)}]$, so we get a contradiction and hence complete the proof of our claim.

\textbf{5.3. The indexes of the quadratic Eisenstein ideals.} In the following, we will denote by $\BT$ to be the full Hecke algebra $\BT_0(DC)$ of level $\Gamma_0(DC)$ generated over $\BZ$ by all the $T_\ell=T^{\Gamma_0(DC)}_\ell$ for all the primes $\ell$.

\begin{lem}\label{lem5}
For any quadratic character $\psi$ of conductor $f$, there is a natural isomorphism
\[{\BT}/{I_{\Gamma_0(DC)}(E_{M,L,\psi})}\simeq{\BZ}/{m\BZ},\]
for some non-zero integer $m$.
\end{lem}
\begin{proof}
It is obvious that the natural homomorphism $\BZ\rightarrow{\BT}/{I_{\Gamma_0(DC)}(E_{M,L,\psi})}$ is surjective, so we only need to prove that the kernel of this homomorphism is non-zero. However, suppose the kernel is zero so that $\BZ\simeq{\BT}/{I_{\Gamma_0(DC)}(E_{M,L,\psi})}$, then the ring homomorphism $\BT\rightarrow{\BT}/{I_{\Gamma_0(DC)}(E_{M,L,\psi})}\simeq\BZ\hookrightarrow\BC$ gives rise to a normalized cusp form whose eigenvalue is $\psi(\ell)+\ell\cdot\psi(\ell)$ for any $\ell\nmid D$, which contradicts the Ramanujan bound. Thus the kernel must be of the form $(m)$ for some non-zero integer $m$ and we have hence proved the lemma.
\end{proof}

\begin{prop}\label{index}
For any quadratic character $\psi$, there is a natural isomorphism
\[{\BT}/{I_{\Gamma_0(DC)}(E_{M,L,\psi})}\otimes\BZ[\frac{1}{6D}]\simeq C_{\Gamma_0(DC)}(E_{M,L,\psi})\otimes\BZ[\frac{1}{6D}],\]
which is induced from the action of $\BT$ on the cuspidal group $C_{\Gamma_0(DC)}(E_{M,L,\psi})$.
\end{prop}
\begin{proof}
Recall that there is a perfect pairing of $\BZ$-modules (see \cite{Ri})
\[\BT\times S_2(\Gamma_0(DC),\BZ)\rightarrow\BZ,\]
which maps any $(T,f)$ to $a_1(f|T;[\infty])$. Tensor with ${\BZ}/{m\BZ}$ over $\BZ$, we get another perfect pairing
\[{\BT}/{m\BT}\times S_2(\Gamma_0(DC),{\BZ}/{m\BZ})\rightarrow{\BZ}/{m\BZ},\]
where $m$ is the non-zero integer in Lemma~\ref{lem5}. Because ${\BT}/{I_{\Gamma_0(DC)}(E_{M,L,\psi})}$ is a quotient of ${\BT}/{m\BT}$, it follows that there is a perfect pairing
\[{\BT}/{I_{\Gamma_0(DC)}(E_{M,L,\psi})}\times S_2(\Gamma_0(DC),{\BZ}/{m\BZ})[I_{\Gamma_0(DC)}(E_{M,L,\psi})]\rightarrow\frac{\BZ}{m\BZ}\]
of ${\BZ}/{m\BZ}$-modules, and hence we get a canonical isomorphism
\[S_2(\Gamma_0(DC),{\BZ}/{m\BZ})[I(E_{M,L})]\simeq{\BZ}/{m\BZ},\]
which gives us a unique normalized cusp form $F\in S_2(\Gamma_0(DC),\BZ)$ such that $F\equiv E_{M,L,\psi}\pmod m$
In other words, there exists some $G\in M_2(\Gamma_0(DC),\BZ)$ such that $F=E_{M,L,\psi}+m\cdot G$. However, by Theorem 1.6.2 of \cite{K}, the constant terms of $G$ at the cusps are all in $\BZ[\frac{1}{6D},\mu_D]$, so we find that
\[\varphi(\frac{D}{f})\cdot\mu(\frac{L}{f})\in m\cdot\BZ[\frac{1}{6D},\mu_D]\bigcap\BQ=m\cdot\BZ[\frac{1}{6D}]\]
by Proposition~\ref{constant3} which gives the explicit values of the constant terms of $E_{M,L\psi}$. On the other hand, since $C_{\Gamma_0(DC)}(E_{M,L,\psi})$ is cyclic,it follows that $\frac{\BZ}{m\BZ}\simeq\frac{\BT}{I(E_{M,L})}$ acts transitively on it, so that
\[m\in\varphi(\frac{D}{f})\cdot\mu(\frac{L}{f})\cdot\BZ[\frac{1}{6D}]\]
by Corollary~\ref{order} about the explicit value of the order of $C_{\Gamma_0(DC)}(E_{M,L,\psi})$. We have thus completed the proof of the theorem.
\end{proof}

\begin{remark} When combined with Corollary~\ref{order} about the order of the quadratic cuspidal groups, the above theorem also give the index of the quadratic Eisenstein ideals in $\BT$ up to a factor of $6D$.
\end{remark}
\textbf{5.4. Proof of Theorem~\ref{M2}.}
For any $f\mid C$, let $\psi$ be the unique quadratic character of conductor $f$. Recall that
\[J_0(DC)(\psi):=\{P\in J_0(DC)(\overline{\BQ}):\ \sigma(P)=\psi(\sigma)\cdot P \text{ for any }\sigma\in G_\BQ\}\]
We \textbf{claim} that, for any prime $q$ not dividing $6\cdot D\cdot\varpi(D)$,
\[J_0(DC)(\psi)[q^\infty]=0,\]
which of course implies Theorem~\ref{M2}. Since any positive divisor of $DC$ is of the form $dc$ with $1\leq c\mid d\mid D$ and $c\mid C$, the commutative diagram defining the new part of $J_0(DC)$ can be written as
\[\xymatrix{
  J^{new}_0(DC) \ar[d] \ar[r]
                & 0 \ar[d]  \\
  J_0(DC)  \ar[r]
                & \prod_{1<\alpha\mid\frac{DC}{dc}}J_0(dc)}.\]

\begin{lem}\label{lem6}
If $f\nmid c$, then $J_0(dc)(\psi)[q^\infty]=0$.
\end{lem}
\begin{proof}
Firstly, if $f\nmid d$, then $J_0(dc)$ has good reduction at any prime divisor $p$ of $f$ not dividing $d$. It follows that $_0(dc)[q^\infty]$ is unramified at $p$. But $p\mid f$ implies that $\psi$ is ramified at $p$, so that $J_0(dc)(\psi)[q^\infty]$ must be zero.

On the other hand, if $f\mid d$ but $f\nmid c$. Let $p$ be a prime divisor of $f$ not dividing $c$. Then $J)(dc)$ has semi-stable reduction at $p$, so the inertia group $I_p$ acts unipotently on $T_q(J_0(dc))$. If $PJ_0(dc)(\psi)[q^\infty]$, then $(1-\sigma)^k(P)=0$ for any $\sigma\in I_p$ with $k$ some positive integer. But there is some $\sigma\in I_p$ such that $\sigma(P)=\psi(\sigma)\cdot P=-P$ as $p\mid f$, so that $2^k\cdot P=0$ for some $k$ which contradicts the assumption that $q\neq2$. We have thus finished the proof of the lemma.
\end{proof}

\begin{lem}\label{lem7}
$J^{new}_0(DC)(\psi)[q^\infty]=0$.
\end{lem}
\begin{proof}

By Eichler-Shimura theory, for any prime $\ell\nmid D$, $T_\ell$ acts as multiplication by $\psi(\ell)+\ell\cdot\psi(\ell)$ on $J^{new}_0(DC)(\psi)[q^\infty]$. On the other hand, the new form theory tells us that $T_\ell$ acts as $\pm1$ if $\ell\mid\frac{D}{C}$, and $T_\ell$ acts as $0$ if $\ell\mid C$. Thus, if $J^{new}_0(DC)(\psi)[q^\infty]\neq0$, then
\[S_2(\Gamma_0(DC),\BF_q)\left[\{T_\ell-(\psi(\ell)+\ell\cdot\psi(\ell))\}_{\ell\nmid D},\{T_\ell\}_{\ell\mid C},\{T_\ell-\delta_\ell\}_{\ell\mid\frac{D}{C}}\right]\neq0\]
and is generated by a unique normalized $\theta$. Here, for any $\ell\mid\frac{D}{Cp}$, $\delta_\ell=\pm1$ according to how $T_\ell$ acts. However, simple manipulation on Fourier expansions shows that

$\bullet$ If $\delta_\ell=1$, then $[\ell]^-(\theta)$ belongs to $S_2(\Gamma_0(DC\ell),\BF_q)$ and is annihilated by $T_\ell$;

$\bullet$ If $\delta_\ell=-1$, then $\theta+\frac{1}{\ell}\theta|\gamma_\ell$ belongs to $S_2(\Gamma_0(DC\ell),\BF_q)$ and is also annihilated by $T_\ell$.

It follows that, by raising the levels in such a way, we will finally get some normalized form which spans the one-dimensional $\BF_q$-vector space
\[S_2(\Gamma_0(D^2),\BF_q)[\{T_\ell-(\psi(\ell)+\ell\cdot\psi(\ell))\}_{\ell\nmid D},\{T_\ell\}_{\ell\mid D}]\]
Since the ideal $(\{T_\ell-(\psi(\ell)+\ell\cdot\psi(\ell))\}_{\ell\nmid D},\{T_\ell\}_{\ell\mid D})$ is exactly the Eisenstein ideal $I_{\Gamma_0(DC)}(E_{D,D,\psi})$, we find that $q$ divides the index of $I_{\Gamma_0(D^2)}(E_{D,D,\psi})$ in $\BT_0(D^2)$. By Proposition~\ref{index}, it follows that $q$ divides the order of $C_{\Gamma_0(D^2)}(E_{D,D,\psi})$ as we have assumed that $q\nmid 6D$. But because $q\nmid\varphi(D)\cdot\mu(D)$, it is clear from Theorem~\ref{order} that $C_{\Gamma_0(D^2)}(E_{D,D,\psi})[q^\infty]=0$, so we get a contradiction and hence completes the proof.
\end{proof}

\textbf{Proof of the claim:} Firstly, we prove that $J_0(f^2)(\psi)[q^\infty]=0$. By Lemma~\ref{lem6}, $J_0(nm)(\psi)[q^\infty]$ is zero for any $1\leq n\mid m\mid f$ with $mn\neq f^2$. Moreover, by applying Lemma~\ref{lem7} to the situation when $DC=f^2$, we find that $J^{new}_0(f^2)(\psi)[q^\infty]$ is also zero. It follows that $J_0(f^2)(\psi)[q^\infty]=0$. In general, by induction hypothesis, we have $J_0(dc)(\psi)[q^\infty]=0$ for any $1\leq c\mid d\mid D$ with $dc\neq DC$. Then, it follows that $J_0(DC)(\psi)[q^\infty]=J^{new}_0(DC)(\psi)[q^\infty]$, which is zero by Lemma~\ref{lem7}. We have thus complete the proof of the claim and hence that of Theorem~\ref{M2}.

\begin{remark}
To have a complete understanding of these Hecke module structures, it seems that a deeper study of the arithmetic-geometric properties of $X_0(DC)$ is required. Moreover, from the previous results, it is curious to ask whether there is also an intrinsic characterization of the whole cuspidal subgroup $C_0(N)$ in the spirit of generalized Ogg's conjecture. More precisely, we can ask whether the following is true
\[J_0(N)(\BQ_N)_{tor}=C_0(N),\]
where $\BQ_N:=\bigcup_{1\leq d\mid N}\BQ(\mu_{(d,\frac{N}{d})})$. We will study this question in the future.
\end{remark}

\section{Appendix}
In this appendix, we complete the computations of the $2$-part of $C_{\Gamma_0(DC)}(E_{M,L})$ when $D$ is odd. We will need some basic properties of Dedekind sums which we will now briefly recall. The reader is recommend to \cite{R-G} for the details. For any two integers $h,k$ with $k\geq1$ and $(h,k)=1$, the associated \emph{Dedekind sum} is defined to be
\[s(h,k):=\sum^k_{\mu=1}((\frac{h\mu}{k}))((\frac{\mu}{k}))\]
where $((x))$ is defined to be
\[((x))=
\begin{cases}
0 &,\text{\ if}\ x\in\BZ\\
x-[x]-\frac{1}{2} &,\text{}\ otherwise
\end{cases}
\]
for any real number $x$. The famous \emph{reciprocity formulas} for these Dedekind sums says that
\begin{align}
\ s(h,k)+s(k,h)=-\frac{1}{4}+\frac{1}{12}\left(\frac{h}{k}+\frac{1}{hk}+\frac{k}{h}\right)
\end{align}
for any two positive integers $h,k$ with $(h,k)=1$. More over, for any $\gamma=\left(
                  \begin{array}{cc}
                    a & b \\
                    c & d \\
                  \end{array}
                \right)\in SL_2(\BZ)
$, we have that (see \cite{R-G}, P48)
\begin{align*} \begin{split}
  &\log\eta(\gamma z)-\log\eta(z)\\
  &=\frac{1}{2}\cdot sgn(c)^2\cdot\log\left(\frac{cz+d}{i\cdot sgn(c)}\right)+\pi i\cdot\frac{a+d}{12c}-\pi i\cdot sgn(c)\cdot s(d,|c|)
\end{split} \end{align*}
where $\eta$ is the Dedekind $\eta$-function, $sgn(c)$ equals the sign of $c$ if $c\neq0$ and is defined to be zero if $c=0$. If we define a function $\Phi$ on $\SL_2(\BZ)$ as
\begin{align}
\Phi(\gamma):=
\begin{cases}
b/d &,\text{if}\ c=0\\
\frac{a+d}{c}-12\cdot sgn(c)\cdot s(d,|c|) &,\text{if}\ c\neq0
\end{cases}
\end{align}
for any $\gamma=\left(
                  \begin{array}{cc}
                    a & b \\
                    c & d \\
                  \end{array}
                \right)\in SL_2(\BZ)
$, then we can also write the above transformation formulas as
\begin{align} \begin{split}
  &\log\eta(\gamma z)-\log\eta(z)\\
  &=\frac{1}{2}\cdot sgn(c)^2\cdot\log\left(\frac{cz+d}{i\cdot sgn(c)}\right)+\frac{\pi i}{12}\cdot\Phi(\gamma)
\end{split} \end{align}
Finally, if $k$ is an \emph{odd positive} integer, then we have the following congruence equation (\cite{R-G}, P37)
\begin{align}
  12\cdot k\cdot s(h,k)\equiv k+1-2(\frac{h}{k})\pmod8
\end{align}
which is useful in studying the periods of some Eisenstein series in $\CE_2(\Gamma_0(N),\BZ)$ as we will see in later sections.

\begin{lem}\label{lem4}For any $1\neq M\mid D$, we have that
\[\int^{\gamma z}_zE_{M,D/M}(\tau)d\tau=\frac{1}{24}\sum_{1\leq r\mid D}(-1)^{\nu(r)-1}\frac{1}{(r,\frac{D}{M})}\Phi\left(
                                                                                           \begin{array}{cc}
                                                                                             a & rb \\
                                                                                             \frac{c}{r} & d \\
                                                                                           \end{array}
                                                                                         \right)
\]
with any $\gamma=\left(
                  \begin{array}{cc}
                    a & b \\
                    c & d \\
                  \end{array}
                \right)\in\Gamma_0(DC)
$ and $z\in\fH$.
\end{lem}
\begin{proof}
We will firstly consider the Eisenstein series $E_{D,1}$. When $\nu(D)=1$, so that $D=p$ for some prime and $C=1$, then
\begin{align*}
  E_{p,1}(z)&=\frac{1}{2}\left[p\cdot\phi_{(0,0)}(pz)-\phi_{(0,0)}(z)\right]\\
  &=\frac{1}{2\pi i}\frac{d}{dz}\left(\log\eta(pz)-\log\eta(z)\right)
\end{align*}
because $(2\pi i)\cdot\phi_{(0,0)}=\frac{1}{z-\overline{z}}+2\frac{d}{dz}\log\eta$ by \cite{St}, Remark 2.4.3. It follows that
\begin{align*}
\int^{\gamma z}_zE_{p,1}(\tau)d\tau&=\frac{1}{2\pi i}\left[\frac{d}{dz}\left(\log\eta(p\gamma z)-\log\eta(\gamma z)\right)-\frac{d}{dz}\left(\log\eta(pz)-\log\eta(z)\right)\right]\\
&=\frac{1}{2\pi i}\left[\frac{d}{dz}\left(\log\eta(\gamma_pp\gamma\gamma^{-1}_p (pz))-\log\eta(pz)\right)-\frac{d}{dz}\left(\log\eta(\gamma z)-\log\eta(z)\right)\right]\\
&=\frac{1}{24}\left[\Phi(\gamma_p\gamma\gamma^{-1}_p)-\Phi(\gamma)\right]
\end{align*}
which is the desired in this special situation. However, if $\nu(D)>1$, then we choose an arbitrary prime divisor $p$ of $D$ and find inductively that
\begin{align*}
 \int^{\gamma z}_zE_{D,1}(\tau)d\tau&=\int^{\gamma z}_zE_{D/p,1}(\tau)d\tau-\int^{\gamma z}_z(E_{D/p,1}|\gamma_p)(\tau)d\tau\\
 &=\int^{\gamma z}_zE_{D/p,1}(\tau)d\tau-\int^{\gamma_p\gamma\gamma^{-1}_p(pz)}_{pz}E_{D/p,1}(\tau)d\tau\\
 &=\frac{1}{24}\sum_{1\leq r\mid D/p}(-1)^{\nu(r)-1}\Phi\left(
                                                                                           \begin{array}{cc}
                                                                                             a & rb \\
                                                                                             \frac{c}{r} & d \\
                                                                                           \end{array}
                                                                                         \right)-\frac{1}{24}\sum_{1\leq s\mid D/p}(-1)^{\nu(r)-1}\Phi\left(
                                                                                           \begin{array}{cc}
                                                                                             a & spb \\
                                                                                             \frac{c}{sp} & d \\
                                                                                           \end{array}
                                                                                         \right)\\
 &=\frac{1}{24}\sum_{1\leq r\mid D}(-1)^{\nu(r)-1}\Phi\left(
                                                                                           \begin{array}{cc}
                                                                                             a & rb \\
                                                                                             \frac{c}{r} & d \\
                                                                                           \end{array}
                                                                                         \right)
\end{align*}
for any $\gamma=\left(
                  \begin{array}{cc}
                    a & b \\
                    c & d \\
                  \end{array}
                \right)\in\Gamma_0(DC)
$. This completes the proof for the Eisenstein series $E_{D,1}$. The proof for more general $E_{M,D/M}$ is similar, in which one precede inductively on $\nu(\frac{D}{M})$ as following
\begin{align*}
&\int^{\gamma z}_zE_{M,D/M}(\tau)d\tau\\
&=\int^{\gamma z}_zE_{M,D/Mp}(\tau)d\tau-\frac{1}{p}\int^{\gamma z}_zE_{M,D/Mp}(p\tau)dp\tau\\
&=\frac{1}{24}\sum_{1\leq r\mid D/p}(-1)^{\nu(r)-1}\frac{1}{(r,\frac{D}{Mp})}\Phi\left(
                                                                                           \begin{array}{cc}
                                                                                             a & rb \\
                                                                                             \frac{c}{r} & d \\
                                                                                           \end{array}
                                                                                         \right)-\frac{1}{24}\sum_{1\leq r\mid D/p}(-1)^{\nu(r)-1}\frac{1}{p(r,\frac{D}{Mp})}\Phi\left(
                                                                                           \begin{array}{cc}
                                                                                             a & prb \\
                                                                                             \frac{c}{rp} & d \\
                                                                                           \end{array}
                                                                                         \right)\\
&=\frac{1}{24}\sum_{1\leq r\mid D}(-1)^{\nu(r)-1}\frac{1}{(r,\frac{D}{M})}\Phi\left(
                                                                                           \begin{array}{cc}
                                                                                             a & rb \\
                                                                                             \frac{c}{r} & d \\
                                                                                           \end{array}
                                                                                         \right)
\end{align*}
\end{proof}

In the following, we denote $\xi_{M,D/M}(\gamma)$ to be $\sum_{1\leq r\mid D}(-1)^{\nu(r)-1}\frac{1}{(r,\frac{D}{M})}\Phi\left(
                                                                                           \begin{array}{cc}
                                                                                             a & rb \\
                                                                                             \frac{c}{r} & d \\
                                                                                           \end{array}
                                                                                         \right)$ for any $\gamma=\left(
                  \begin{array}{cc}
                    a & b \\
                    c & d \\
                  \end{array}
                \right)\in\Gamma_0(DC)
$. Now we can finally prove the first part of Theorem~\ref{M2}

\begin{thm}\label{thm3}
Notations are as above, then $C(E_{M,L})$ is a finite cyclic abelian group. More over, the order $\CN_{M,L}$ is given by the following
\begin{align*}
 \CN_{M,L}:=\begin{cases}
\frac{p-1}{(12,p-1)} &,\text{if}\ DC=p\ for\ some\ prime\ p\\
\frac{\varphi(D)\cdot\mu(L)\cdot(\frac{D}{M},C)}{(24,\varphi(D)\cdot\mu(L)\cdot(\frac{D}{M},C))} &,\text{if}\ otherwise
\end{cases}
\end{align*}
\end{thm}
\begin{proof}
We only need to prove the assertion about its order, as the acyclicity of $C(E_{m,L})$ follows immediately from the definition.

When $D=p$ is a prime and $C$ equals $1$ (or, respectively, $p$), the corresponding assertions about the order of $C_{\Gamma_0(p)}(E_{p,1})$ (respectively, $C_{\Gamma_0(p^2)}(E_{p,p})$) has been verified in \cite{Ogg2} (respectively, \cite{L}), we are thus reduced to consider those $D$ with at least two prime divisors. Since now $\CN_{M,L}$ is nothing but $n_{M,L}$, it follows from Corollary~\ref{order1} that we only need to verify the $2$-part.

Firstly, if $(M,L)\neq1$ and $p$ is a prime divisor of it, then $T_p(E_{M,L})=0$ by Theorem~\ref{Hecke} and so that $C(E_{M,L})$ is also annihilated by $T_p$. But \cite{LO} has proved that $T_p$ acts as multiplication by $p$ on the Shimura subgroup $\sum_{DC}$, and hence $\sum_{DC}\bigcap C(E_{M,L})\subseteq\mu_2$ must be annihilated by multiplication by $p$. Because $p\mid D$ is odd by our assumption, we find the intersection must be zero and hence prove the assertion when $(M,L)\neq1$.

It remains to prove the assertion for those $E_{M,D/M}$'s. We will distinguish into two situations in the following discussion.

(I) Firstly, we consider the Eisenstein series $E_{D,1}$. For any $\gamma=\left(
                  \begin{array}{cc}
                    a & b \\
                    c & d \\
                  \end{array}
                \right)\in\Gamma_0(DC)
$, we have that

(I.1) If $c=0$, then $\xi_{D,1}(\gamma)=\sum_{1\leq r\mid D}(-1)^{\nu(r)-1}\frac{br}{d}=\pm b\cdot(-1)^{\nu(D)-1}\cdot\varphi(D)$, so $\int^{\gamma z}_zE_{D,1}(\tau)d\tau=\frac{\pm b}{24}\varphi(D)\in\CR(E_{D,1})$.

(I.2) If $c$ is odd, then we find by definition (note that we may assume $c>0$)
\begin{align*}
  \xi_{D,1}(\gamma)&=\sum_{1\leq r\mid D}(-1)^{\nu(r)-1}\left(\frac{a+d}{(c/r)}-12\cdot s(d,\frac{c}{r})\right)\\
  &\equiv(-1)^{\nu(D)-1}\cdot\frac{a+d-1}{c}\cdot\varphi(D)-\frac{2}{c}(\frac{d}{c})\prod_{p\mid D}(1-(\frac{d}{p})p)\pmod8\\
  &\equiv(-1)^{\nu(D)-1}\cdot\frac{a+d-1}{c}\cdot\varphi(D)\pmod8
\end{align*}
with the last equality holds because $D$ is odd and $\nu(D)>1$. We have thus prove that $\int^{\gamma z}_zE_{D,1}(\tau)d\tau\in\BZ_2+\frac{\varphi(D)}{24}\BZ_2$ for any such $\gamma$.

(I.3) If $c\neq0$ is even, then $d$ is odd and we may assume $d>0$, so that
\begin{align*}
  \xi_{D,1}(\gamma)&=\sum_{1\leq r\mid D}(-1)^{\nu(r)-1}\left(\frac{a+d}{(c/r)}-12\cdot sgn(c)\cdot s(d,|\frac{c}{r}|)\right)
\end{align*}
By the reciprocity law, we have
\[s(d,|\frac{c}{r}|)+s(|\frac{c}{r},d|)=-\frac{1}{4}+\frac{1}{12}\left(\frac{d}{|c|}r+\frac{r}{c|d|}+\frac{|c|}{dr}\right)\]
It follow that
\begin{align*}
  \xi_{D,1}(\gamma)&\equiv\left(\sum_{1\leq r\mid D}(-1)^{\nu(r)-1}12\cdot sgn(c)\cdot s(|\frac{c}{r}|,d)\right)-\left(\sum_{1\leq r\mid D}(-1)^{\nu(r)-1}\cdot\frac{c}{dr}\right)\pmod8\\
  &\equiv\frac{2}{d}(\frac{|c|}{d})\cdot\sgn(c)\cdot\prod_{p\mid D}(1-\frac{p}{d})+\frac{c}{dD}\cdot\varphi(D)\pmod8\\
  &\equiv\frac{c}{dD}\cdot\varphi(D)\pmod8
\end{align*}
with the last equality holds because $\nu(D)>1$. We have thus prove that $\int^{\gamma z}_zE_{D,1}(\tau)d\tau\in\BZ_2+\frac{\varphi(D)}{24}\BZ_2$ for any such $\gamma$.

It follows that $\int^{\gamma z}_zE_{D,1}(\tau)d\tau\in\BZ_2+\frac{\varphi(D)}{24}\BZ_2$ for any such $\gamma\in\Gamma_0(DC)$. But as \[\CP(E_{D,1})\supseteq\CP_{\Gamma_1(DC)}(E_{D,1})=\BZ+\frac{\varphi(D)}{24}\BZ\]
we find that
\[\CP(E_{D,1})\otimes\BZ_2=\BZ_2+\frac{\varphi(D)}{24}\BZ_2\]
and hence complete the proof for $E_{D,1}$

(II) Now we consider those $E_{M,D/M}$ with $\frac{D}{M}\neq1$. The proof is similar as above. For any $\gamma=\left(
                  \begin{array}{cc}
                    a & b \\
                    c & d \\
                  \end{array}
                \right)\in\Gamma_0(DC)
$, we have that

(II.1) If $c=0$, then
\begin{align*}
\xi_{M,D/M}(\gamma)&=\pm\sum_{1\leq s\mid\frac{D}{M}}\frac{(-1)^{\nu(s)}}{s}\sum_{1\leq t\mid M}(-1)^{\nu(t)-1}tsb\\
&=(\pm b)\sum_{1\leq s\mid\frac{D}{M}}(-1)^{\nu(s)}\sum_{1\leq t\mid M}(-1)^{\nu(t)-1}t=0
\end{align*}

(II.2) If $c$ is odd, then we find by definition (note that we may assume $c>0$)
\begin{align*}
  \xi_{M,D/M}(\gamma)&=\sum_{1\leq s\mid\frac{D}{M}}\frac{(-1)^{\nu(s)}}{s}\sum_{1\leq t\mid M}(-1)^{\nu(t)-1}\left(\frac{a+d}{(c/ts)}-12\cdot s(d,\frac{c}{ts})\right)\\
  &\equiv-\sum_{1\leq s\mid\frac{D}{M}}\frac{(-1)^{\nu(s)}}{s}\sum_{1\leq t\mid M}(-1)^{\nu(t)-1}\frac{ts}{c}\left(\frac{c}{ts}+1-2(\frac{d}{c})(\frac{d}{ts})\right)\pmod8\\
  &\equiv-\frac{2}{c}(\frac{d}{c})\prod_{p\mid\frac{D}{M}}(1-(\frac{d}{p}))\prod_{p\mid M}(1-p(\frac{d}{p}))\equiv0\pmod8
\end{align*}
with the last equality holds because $D$ is odd and $\nu(D)>1$.

(II.3) If $c\neq0$ is even, then $d$ is odd and we may assume $d>0$. Similarly as before, a straight forward calculation by using the reciprocity law show that
\begin{align*}
  \xi_{M,D/M}(\gamma)&\equiv\pm\frac{2}{d}(\frac{|c|}{d})\prod_{p\mid\frac{D}{M}}(1-\frac{1}{p}(\frac{p}{d}))\prod_{p\mid M}(1-(\frac{p}{d}))\equiv0\pmod8
\end{align*}
with the last equality holds because $\nu(D)>1$. We have thus prove that $\int^{\gamma z}_zE_{M,D/M}(\tau)d\tau\in\BZ_2$ for any $\gamma\in\Gamma_0(DC)$ and hence completes the proof of the theorem.
\end{proof}

\end{document}